\documentclass[onefignum,onetabnum]{siamart190516}

\usepackage{bbold}
\usepackage{lipsum}
\usepackage{amsfonts}
\usepackage{graphicx}
\usepackage{epstopdf}
\usepackage{todonotes}
\usepackage{algorithmic}
\ifpdf
  \DeclareGraphicsExtensions{.eps,.pdf,.png,.jpg}
\else
  \DeclareGraphicsExtensions{.eps}
\fi

\newsiamremark{remark}{Remark}
\newsiamremark{hypothesis}{Hypothesis}
\crefname{hypothesis}{Hypothesis}{Hypotheses}
\newsiamthm{claim}{Claim}

\headers{Krylov-Simplex}{W. Vanroose and J. Cornelis}

\title{Krylov-Simplex method that minimizes the residual in $\ell_1$-norm or $\ell_\infty$-norm. \thanks{Draft version of 27 jan 2021}}
\author{Wim Vanroose\thanks{U. Antwerpen, Middelheimlaan 1, 2020
    Antwerpen, Belgium (\email{wim.vanroose@uantwerpen.be})} \and Jeffrey Cornelis\thanks{U. Antwerpen,
    Middelheimlaan 1, 2020 Antwerpen, Belgium (\email{jeffrey.cornelis@uantwerpen.be})}}

\usepackage{amsopn}

\begin{document}

\maketitle

\begin{abstract}
  The paper presents two variants of a Krylov-Simplex iterative method
  that combines Krylov and simplex iterations to minimize the residual
  $r = b-Ax$. The first method minimizes $\|r\|_\infty$, i.e. maximum
  of the absolute residuals. The second minimizes $\|r\|_1$, and finds
  the solution with the least absolute residuals.  Both methods search
  for an optimal solution $x_k$ in a Krylov subspace which results in
  a small linear programming problem. A specialized simplex algorithm
  solves this projected problem and finds the optimal linear
  combination of Krylov basis vectors to approximate the solution.
  The resulting simplex algorithm requires the solution of a series of
  small dense linear systems that only differ by rank-one updates.
  The $QR$ factorization of these matrices is updated each
  iteration. We demonstrate the effectiveness of the methods with
  numerical experiments.
\end{abstract}




\section{Introduction}
Given vectors $x,y \in \mathbb{R}^n$ the distance between them in the
$\ell_\infty$-norm is $\|x-y\|_\infty = \max_{i=1}^{n} |x_i-y_i|$. It
is the \textit{chessboard} distance or Chebyshev distance. It
is the largest difference along any of the coordinate axes. Their distance
in the $\ell_1$-norm is $\|x-y\|_1 = \sum_{i=1}^{n} |x_i-y_i|$. This
is known as the \textit{Manhattan} distance. It is the sum of the
lengths along these coordinate axes.

In this paper we minimize the residual $\|b-Ax_k\|_1$, known as
\textit{least absolute residuals} or \textit{least absolute
  deviations}(LAD), and $\|b-Ax_k \|_\infty$-norm, known as
\textit{least maximal absolute residual}, for a linear system
$A\in\mathbb{R}^{m \times n}$ and $b \in \mathbb{R}^n$ where $x_k \in
x_0 + \mathcal{K}_k$ is chosen from a Krylov subspace
$\mathcal{K}_k$.

Minimizing this residual in the $\ell_1$- and $\ell_\infty$-norm
promotes certain properties in the residual and the solution. The
$\ell_1$-norm is well known to promote sparsity. In many applications
such as reconstruction in inverse problems or compressed sensing the
solutions are sparse and by using the $\ell_1$ norm this property is
promoted in the solution \cite{donoho2006compressed,
  tibshirani1996regression}. In contrast, the $\ell_\infty$ norm
penalizes the largest deviations what is useful in model calibration
where, for example, the maximal deviation between the model and the
data is minimized.  For example, in optical lithography the optimal
settings of a lens system are such that the maximal absolute deviation
over the whole optical field is minimized. While minimizing the average
deviation using $\ell_2$-norm still allows large local deviations
resulting in faulty transistors.

Krylov subspace methods solve sparse linear algebra problems by
finding coefficients $y_k \in \mathbb{R}^k$ that describe the solution
as ${x_k = x_0 + V_k y_k}$ and the residual as ${r_k = r_0- AV_ky_k}$,
where $V_k$ is a basis for $\mathcal{K}_k$.  The sparse matrix-vector
products $Av$ and the required dot-products to generate this basis
$V_k$ can be easily parallelized over many computer cores.  Combined
with preconditioners Krylov methods can solve systems with millions of
unknowns in a scalable way \cite{saad,petsc-user-ref}.
    
Minimizing the residual of a square matrix over a Krylov subspace in
the $\ell_2$-norm leads to a small least squares problem,
${\min_{y_k} \left\|H_{k+1,k}y_k - \|r_0\|e_1 \right\|_2}$, where
$H_{k+1,k}$ is a small Hessenberg matrix.  This small projected
problem is solved with Givens rotations \cite{liesen2013krylov}.  This
is known as \textit{Generalized Minimal Residual}(GMRES) and requires
the explicit storage of the basis vectors \cite{saad}.
  
Similarly, minimizing the $A$-norm of the error $\|e\|_A$ of positive
definite matrix leads to  ${T_{kk}y - \|r_0\|_2 e_1 = 0}$, where
$T_{kk}$ is the tri-diagonal Lanczos matrix, which can be solved with
recurrence relations. This method is known as \textit{Conjugate
  Gradients} (CG) \cite{hestenes}.

The optimization problem can be formulated as an LP problem and solved
with state of the art LP solvers such as Highs \cite{highs} that use a
simplex that exploits the sparsity and hypersparsity appearing in the
problem.

Fountoulakis, Gondzio and collaborators have studied various data
science problem that include $\ell_1$-norms with a primal-dual
interior point method combined with inner preconditioned Krylov
iterations \cite{gondzio, fountoulakis, fountoulakis2016,
  dassios2015preconditioner}. Here Newton outer iterations are
combined with Krylov inner iterations.

Problems with $l_1$-norm can also be reformulated as a iterative
reweighted least-squares (IRLS) problem, where the weights are adapted
to each each inner iteration \cite{holland1977robust}.

Luce, Tebbens and collaborators explored Krylov methods to solve the
nucleus that appears after reordering of rows and columns in the
simplex method \cite{luce2009factorization}. They explored a
representative set of MIP problems and concluded that the size of the
nucleus is so small combined with adverse spectral properties that
Krylov methods are not competitive.

In contrast in this paper we use Krylov as an outer iteration and
project the problem by optimizing $\min\|b-Ax\|_1$ and
$\min\|b-Ax \|_\infty$ over a Krylov subspace.  After projection on a
basis $V_k$ these problems leads to a small linear programming problem
with dense matrices.  In the outer iteration the Krylov subspace in
expanded and the inner simplex iteration solve the small LP problem
for the optimal solution within the current Krylov subspace.  This
leads to a sequence of small dense matrices $B$ and $B^T$ that need to
be solved.  Since these matrices appear in a sequence where they
differ only by rank-one updates, we can easily update their $QR$
factorization \cite{golub13, elble2012review}.

The outline of the paper is as follows. In
section~\ref{review_krylov_simplex} we shortly explain the basics of
Krylov subspaces and the revised simplex method.  We then, in
section~\ref{sec:max}, derive the Krylov-Simplex method for residual
minimization in $\ell_\infty$.  In section~\ref{sec:l1} we repeat this
for the $\ell_1$-norm.  In section~\ref{sec:numerical results} we give
some convergence bound and illustrate the method with numerical
applications.  We then discuss the results an draw conclusions.


\section{Basics of Krylov and revised simplex method} \label{review_krylov_simplex}

We summarize some required basics from Krylov subspaces and the primal
revised simplex method.

\subsection{Krylov Subspaces}
For a square matrix $A \in\mathbb{R}^{n \times n}$, a right hand side
$b \in \mathbb{R}^n$ and an initial guess $x_0 \in \mathbb{R}^n$,  a
Krylov subspace is
\begin{equation}
  \mathcal{K}_k(A,r_0) = \text{span} \{r_0, Ar_0, A^2r_0, \ldots , A^{k-1}r_0 \},
\end{equation}
with $r_0= b-Ax_0$.  With the help of the Arnoldi
routine, algorithm \ref{arnoldi_algorithm}, we can make an orthogonal
basis $V_k$ and a Hessenberg matrix $H_{k+1,k} \in \mathbb{R}^{k+1
  \times k}$ such that
\begin{equation}
  AV_k = V_{k+1} H_{k+1,k}.
\end{equation}
We then represent the solution as $x_k = x_0 + V_k y_k$.
  
For a non-square matrix $A \in \mathbb{R}^{m \times n}$, a right hand
side $b \in \mathbb{R}^m$ and an initial guess $x_0 \in \mathbb{R}^n$,
a Krylov subspace is then
\begin{equation}
 \mathcal{K}_k(A^TA,A^Tr_0) = \text{span} \{A^T r_0,(A^T A) A^T
 r_0, (A^T A)^2 A^T r_0, \ldots , (A^TA)^{k-1}A^Tr_0 \}.
\end{equation}
The Golub-Kahan  \cite{golub1965} relation is
  \begin{equation}
    A V_{k}  = U_{k+1} B_{{k+1},k},
  \end{equation}
where $B_{k+1,k} \in \mathbb{R}^{{k+1}\times k}$ is a bi-diagonal
matrix containing the $\alpha_k$ on the diagonal and $\beta_k$ on the
first upper diagonal. They are calculated in algorithm
\ref{golubkahan_algorithm}.  The resulting basis $V_k$ is used to
represent the solution as $x_k = x_0 + V_k y_k$.

\begin{tabular}{cc}
  \begin{minipage}{0.4\textwidth}
    \begin{algorithm}[H]    
      \begin{algorithmic}[1]
        \STATE{$r_0 = b-Ax_0$}
        \STATE{$v_1 = r_0/\|r_0\|_2$}
        \FOR{$k=1,2,\ldots,$}
        \STATE{$ w = Av_i$}
        \FOR{$j=1,\ldots,k$}
        \STATE{$h_{jk} = v_j^T w$}
        \STATE{$w = w- h_{jk} v_k$}
        \ENDFOR
        \STATE{$h_{k+1,j} = \|w\|_2$}
        \STATE{$v_{k+1} = w/h_{k+1,k}$}
        \ENDFOR
      \end{algorithmic}
      \caption{Arnoldi Iteration \label{arnoldi_algorithm}}
    \end{algorithm}
  \end{minipage}
  &
  \begin{minipage}{0.4\textwidth}
    \begin{algorithm}[H]    
      \begin{algorithmic}[1]
        \STATE{$r_0 = b-Ax_0$}
        \STATE{$u_1 = r_0/\|r_0\|_2$}
        \STATE{$\beta_0=0$}
        \FOR{$k=1,2,\ldots$}
        \STATE{$v_k = A^Tu_k - \beta_{k-1} v_{k-1}$}
        \STATE{$\alpha_k = \|v_k\|_2$}
        \STATE{$v_k = v_k/\alpha_k$}
        \STATE{$u_{k+1} = A v_k -\alpha_k u_k$}
        \STATE{$\beta_{k+1} = \|u_{k+1}\|_2$}
        \STATE{$u_{k+1} = u_{k+1} / \beta_{k+1}$}
        \ENDFOR
      \end{algorithmic}
      \caption{ Golub-Kahan bi-diagonalization     \label{golubkahan_algorithm}}
    \end{algorithm}
\end{minipage}
\end{tabular}
\vspace{0.5cm}

A happy breakdown occurs when the new vector $Av_{k}$ is a linear
combination of the existing basis vectors $V_k$.  Then the Arnoldi
relation becomes $AV_k = V_k H_{k,k}$ with a square Hessenberg
matrix. Similarly, we have $AV_k = U_k B_{k,k}$ with a square matrix
$B_{k,k}$.  This means $V_k$ spans an invariant subspace under $A$ and
the solution can be found in this subspace \cite{liesen2013krylov}.
In practice Krylov methods are always combined with preconditioners to
accelerate the convergence.

For the remainder of the text we will use $\mathcal{K}_k$ to denote
the Krylov subspace. Depending on the dimensions of the matrix this refers
to the Krylov subspace generated by Arnoldi, for square matrices, or
Golub-Kahan, for non-square matrix.

For more background on Krylov subspaces we refer to
\cite{liesen2013krylov, saad}.


\subsection{Revised primal Simplex}
The revised simplex method solves a linear programming problem in its
standard form
\begin{equation}
    \begin{aligned}
      \min &\,c^T x, \\
      \text{s.t.}&\, Ax=b, \\
      &\,x \ge 0,
    \end{aligned}
\end{equation}
where $A \in \mathbb{R}^{m \times n }$, $b \in \mathbb{R}^m$ and $c$
and $x \in \mathbb{R}^n$.  The corresponding Karush-Kuhn-Tucker (KKT)
conditions are
  \begin{equation}
    \begin{aligned}
      A^T \lambda + s = c, \quad    Ax = b, \\
      x \ge 0, \quad   s \ge 0, \\ x_i s_i = 0, \quad \forall i \in \{ 1, \ldots,n\}.
    \end{aligned}
  \end{equation}
The revised simplex method solves this system of equations.
  
It is assumed that $A$ is full rank such that there is a subset of
columns of $A$ that form a non-singular sub-matrix $B \in \mathbb{R}^{m
  \times m}$ with $B^{-1}b \ge 0$.  The indices of these columns are denoted by $\mathcal{B}$
and are called basic. The remaining indices are denoted as
$\mathcal{N} = \{1,\ldots, n\} \setminus \mathcal{B}$ and are called
non-basic. The remaining columns of $A$ are denoted as $N$. The
solution is then split in $x_B = x|_{i \in \mathcal{B}}$ and $x_N =
x|_{i \in \mathcal{N}}$,

In the first step, the initial $B$ is factorized and we solve the primal
equation
\begin{equation}
    Ax = B x_B + N x_N = b,
\end{equation}
as $x_B = B^{-1} b$ and $x_N=0$.  A Phase-I simplex typically finds a
$B$ such that $x_B =B^{-1}b \ge 0$.

In the second step we solve the dual equation, $B^T\lambda = c_B$ and
$N^T \lambda +s_N =c_N$
\begin{equation}
     \lambda = B^{-T} c_B
\end{equation}
reusing the factorization of $B$. We then have $s_N =c_N- N^T
\lambda$, also known as the reduced costs. If all $s_N \ge 0$, we have
found the optimal solution.

Otherwise, we update the indices  of the basic set $\mathcal{B}$. We
select a $q \in \mathcal{N}$ where $s_q <0$ and solve $B d = A_q$ for
the search direction $d$.  We then update
\begin{equation}
    x^+_B = x_B - d \alpha \quad \text{and} \quad x_q = \alpha.
\end{equation}
Since $x$ needs to be non-negative, the step size $\alpha$ is
determined from
\begin{equation}
    \alpha = \min_{i \in \mathcal{B} , d_i >0} (x_B)_i/d_i
\end{equation}
with $r$ denoting the minimizing index.  At this point 
$x_r$ with $r \in \mathcal{B}$ becomes zero and the non-basic $q$ became
non-zero. We change $\mathcal{B}=\mathcal{B} \cup \{q\} \setminus \{
r\}$.

We then go to the second step, avoiding a solve of the primal
equation.  This loop is repeated until all Lagrange multipliers $s_N$,
associated with $x \ge 0$, are positive. 

The matrices $B$ are updated in each iteration depending on the
changes in the basic set $\mathcal{B}$. Since $\mathcal{B}$ only
changes by one index the $B$ matrices only differ by a column. This is a rank-one
update.  The factorization of $B$ can be efficiently updated. Further details on each
step of the revised simplex method and updating the factorization can be
found in \cite{nocedal2006numerical, elble2012review}.


\section{Krylov-Simplex for the $\ell_\infty$-norm} \label{sec:max}
In this section we derive a Krylov-simplex method for the minimization
of $r_k = b-Ax_k$ in the $\ell_\infty$-norm with $x_k \in x_0 +
\mathcal{K}_k$.  This leads to a linear programming (LP) problem for
$y_k$, the coefficients in the basis $V_k$ that spans
$\mathcal{K}_k$.  The projected problem is solved by a specialized primal simplex.

Section \ref{sec:kkt infty} derives the KKT conditions of the
projected LP problem for a given Krylov subspace $\mathcal{K}_k$.
Section \ref{sec:kkt primal dual infty} solves the primal and dual
equation from the KKT conditions for a given basic set $\mathcal{B}_k$. If the
resulting Lagrange multipliers are positive, we have found the optimal
solution for the current Krylov subspace. Otherwise we pivot as
discussed in section \ref{sec:pivot infty}.  Section
\ref{sec:expanding infty} discusses how an in initial
$\mathcal{B}_{k+1}$ is found for $\mathcal{K}_{k+1}$.

\begin{definition}\label{max-norm Krylov}
Let $A \in \mathbb{R}^{m \times n}$ a matrix and $b \in \mathbb{R}^m$
a right hand side.  The iterates of the $\ell_\infty$-norm Krylov-Simplex
are given by
  \begin{equation}
    x_k := \text{argmin}_{ x \in x_0 + \mathcal{K}_k} \|r_k\|_\infty.
  \end{equation}
\end{definition}

The coefficients $y_k$ that describe the iterates of the
$\ell_\infty$-norm Krylov-Simplex for $\mathcal{K}_k$, with a basis
$V_k$, are the solution of the following linear programming problem
\begin{equation}
  \begin{aligned}
    \min &\, \gamma_k \\
 \text{s.t.} &\, -\gamma_k  \mathbb{1}_m \le r_0 - AV_k y_k \le \gamma_k \mathbb{1}_m,
  \end{aligned}
\end{equation}
where the $\gamma_k := \|r_k\|_\infty \in \mathbb{R}$, the max
of the absolute values of the residual at iteration $k$. The
$\mathbb{1}_m$ denotes a column vector with ones of length
$m$. Alternatively,
\begin{equation}
  \begin{aligned}
    \min & \, \gamma_k\\
  \text{s.t.}  &\,  AV_k y_k -\gamma_k \mathbb{1}_m \le r_0,\\
    &\, -AV_k y_k -\gamma_k  \mathbb{1}_m \le -r_0.
  \end{aligned}
\end{equation}
or, in matrix notation
  \begin{equation}\label{LP}
  \begin{aligned}
      \min_{\gamma_k,y_k} &\, \begin{pmatrix}1 & 0 \end{pmatrix}^T \begin{pmatrix}\gamma_k \\ y_k\end{pmatrix}\\
     &\, \begin{pmatrix}
        -\mathbb{1}_m  & AV_k \\
        -\mathbb{1}_m  &-AV_k
      \end{pmatrix}
      \begin{pmatrix}\gamma_k\\ y_k 
      \end{pmatrix}
      \le \begin{pmatrix}
        r_0\\
        -r_0
      \end{pmatrix}.
  \end{aligned}
 \end{equation}
We do not explicitly add the condition $\gamma_k >0$ to the optimization problem. 
  
\begin{figure}
    \begin{center}
      \includegraphics[width=\textwidth]{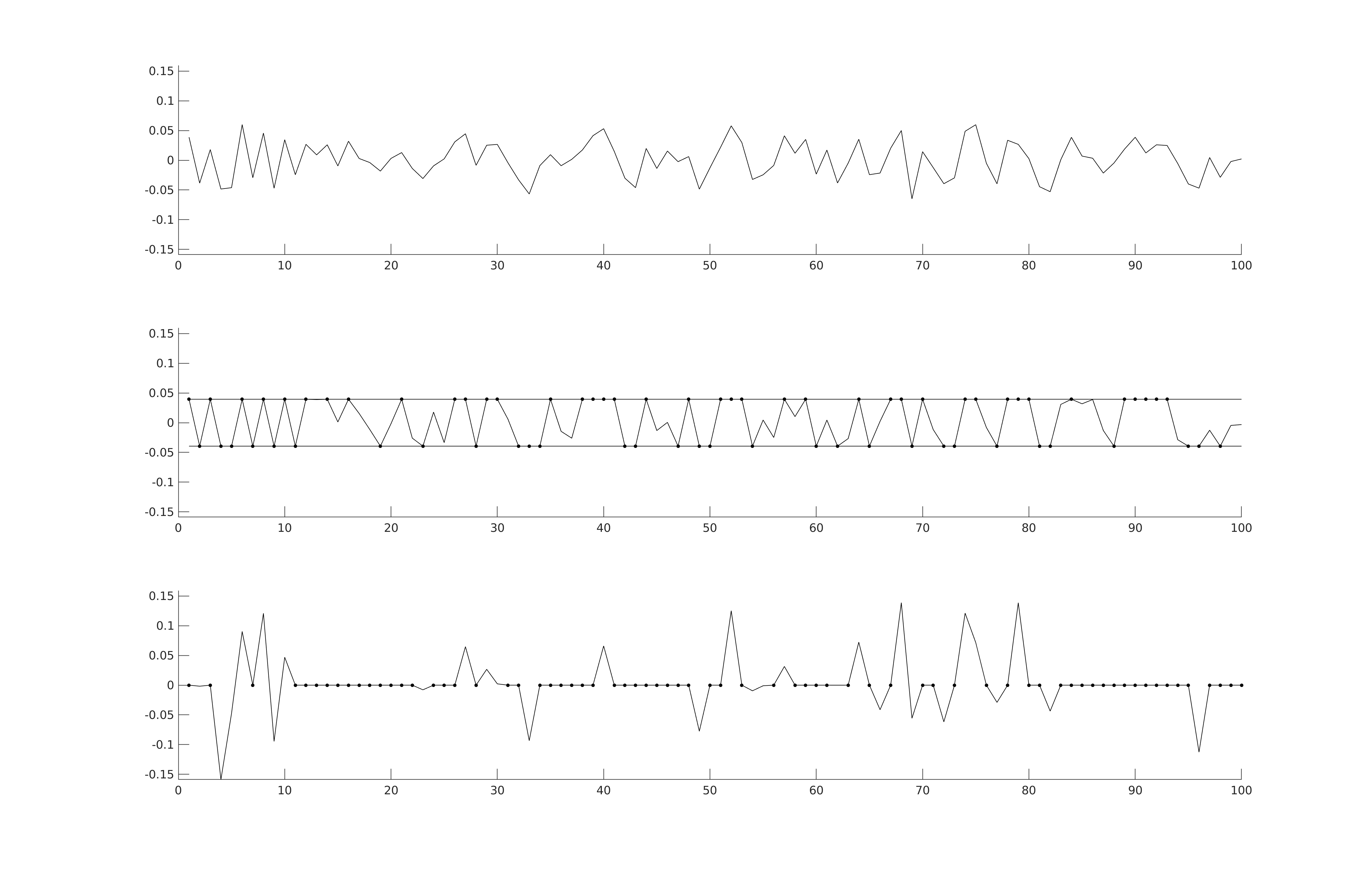}
    \end{center}
    \caption{The optimal residual for a Krylov subspace of size k=70
      in the $\ell_2$ (top), $\ell_\infty$ (middle) and $\ell_1$-norm
      (bottom) for a non-square $100 \times 90$ matrix.  The $\ell_2$
      solution is calculated with GMRES applied to the normal
      equations and has the least squared residual.  The $\ell_\infty$
      solution has the least maximal absolute value of the
      residual. The active constraints are indicated with filled
      circles. The number of residuals that hit the max is equal to
      $k+1$.  The residual that gives minimal $\ell_1$ in the same
      Krylov subspace is shown at the bottom. There are $k$ residuals
      equal to zero, the size of the Krylov subspace.  In this paper
      the optimal subset of residuals that is active for each Krylov
      subspace is found by a Simplex method}
  \end{figure}

  \subsection{KKT conditions} \label{sec:kkt infty}
 In this subsection we derive the KKT conditions for the projected LP
 problem \eqref{LP}. The Lagrangian is
\begin{equation}
  \mathcal{L}(\gamma_k, y_k,\lambda, \mu) = \gamma_k  -  \lambda^T \left( r_0 - AV_k y_k + \gamma_k \mathbb{1}_m\right)  -   \mu^T \left(- r_0 +  AV_k y_k  + \gamma_k \mathbb{1}_m \right).
\end{equation}
This leads to the following KKT conditions:
\begin{eqnarray}
     \mathbb{1}_m^T (\lambda+\mu) &=& 1,  \label{kkt_dual_first}\\
     V_k^T A ^T (\lambda-\mu) &=& 0, \label{kkt_dual_second}\\
     -\gamma_k \mathbb{1}_m  \le r_0 - AV_k y_k  &\le& \gamma_k \mathbb{1}_m  \label{primal_max}\\
    \lambda \ge 0,  \quad   \mu  &\ge& 0,\\
    \lambda_i (r_0 -AV_ky + \gamma_k\mathbb{1}_m)_i &=&0,  \quad \forall      i \in  \{1, \ldots, m \}\label{complementarity_first}\\
    \mu_i (\gamma_k\mathbb{1}_m-(r_0 -AV_ky))_i &=&0,      \quad \forall      i \in  \{1, \ldots, m \}.  \label{complementarity_second}
\end{eqnarray}
The first two equations, \eqref{kkt_dual_first} and
\eqref{kkt_dual_second} are the dual equations, the third equation is
the primal equation and the last two, \eqref{complementarity_first}
and \eqref{complementarity_second} are the complementarity conditions
that couple all equations together.

\begin{remark}
  The second condition leads to $V_k \perp A^T(\lambda-\mu)$. When  we make
  use of the Arnoldi relation this is $V_{k+1} \perp \lambda- \mu$. In other words, the
  vector $\lambda -\mu \notin \mathcal{K}_{k+1}$.
\end{remark}

\subsection{Solving the primal and dual equations of KKT system for Krylov subspace $\mathcal{K}_k$}
\label{sec:kkt primal dual infty}

\begin{definition}
  A basic set $\mathcal{B}_k = \mathcal{B}_{k, <} \cup \mathcal{B}_{k, >}$ for $\mathcal{K}_k$ and $\gamma_k >0$  is a subset of
  $k+1$ indices from $\{1, \ldots, m \}$ such that
  \begin{equation}
    (r_0-AV_ky_k)_i=-\gamma_k \quad \forall i \in \mathcal{B}_{k,<},
  \end{equation}
  and
  \begin{equation}
    (r_0-AV_ky_k)_i=\gamma_k \quad \forall i \in \mathcal{B}_{k,>}.
  \end{equation}
and for all $i \in \mathcal{N}_k := \{1, \ldots, m \} \setminus \mathcal{B}_k$ holds that
  \begin{equation}
    |(r_0-AV_ky_k)_i|\le \gamma_k \quad \forall i \in \mathcal{N}_k
  \end{equation}
  and
  \begin{equation}
    B_k :=   \begin{pmatrix}
      -\mathbb{1}_{\mathcal{B}_{k,<}}  &AV_k |_{\mathcal{B}_{k,<}}  \\
      \mathbb{1}_{\mathcal{B}_{k,<}}   & AV_k |_{\mathcal{B}_{k,>}}
    \end{pmatrix}
  \end{equation}
  with $B_k \in \mathbb{R}^{k+1 \times k+1}$ non-singular. 
\end{definition}
 So we assume that the basic set is determined by the indices of the
 residual whose absolute values are equal to the upper bound
 $\gamma_k$.  We set $\mathcal{B}_0 = \{i\}$, the index of where
 $|(r_0)_i|$ is maximal.  Note that this basic set differs
 significantly from the basic set in the revised simplex method.

We have split $\mathcal{B}_k = \mathcal{B}_{k, <} \cup \mathcal{B}_{k,
  >} $, where for $i \in \mathcal{B}_{k,<}$ holds
$(r_0-AV_ky_k)_i=-\gamma_k$. For $i \in \mathcal{B}_{k,>}$ holds
that $(r_0-AV_ky_k)_i=\gamma_k$.  This leads to the linear system
\begin{equation}\label{lmax_primal}
  \begin{pmatrix}
    -\mathbb{1}_{\mathcal{B}_{k,<}}  &AV_k |_{\mathcal{B}_{k,<}}  \\
    \mathbb{1}_{\mathcal{B}_{k,<}}   & AV_k |_{\mathcal{B}_{k,>}}
  \end{pmatrix}
  \begin{pmatrix}
    \gamma_k\\
    y_k
  \end{pmatrix}
  = \begin{pmatrix}
    r_0 |_{\mathcal{B}_{k,<}}\\
    r_0 |_{\mathcal{B}_{k,>}}
  \end{pmatrix}.
\end{equation}
The system determines the coefficients $y_k$ and the bound
$\|r_k\|_\infty = \gamma_k$.  The $\gamma_k$ and the
coefficients $y_k$ are now primal feasible.  This is system is denoted, using the definition of $B_k$ as
$B_k \begin{pmatrix} \gamma_k \\ y_k
\end{pmatrix} = r_0|_{\mathcal{B}_k}$.

We now solve for the Lagrange multipliers of the basic set
$\mathcal{B}_k$. All multipliers corresponding to the non-basic
variables $\mathcal{N}_k$ are zero, due to the complementarity
conditions \eqref{complementarity_first} and
\eqref{complementarity_second}.  The equations \eqref{kkt_dual_first}
and \eqref{kkt_dual_second} become, for the their non-zero Lagrange
multipliers
\begin{equation}
  \begin{aligned}
    \mathbb{1}^T_{\mathcal{B}_{k, <}} \lambda_{\mathcal{B}_{k, <}} + \mathbb{1}^T_{\mathcal{B}_{k, <}} \mu_{\mathcal{B}_{k, <}} = 1,\\
    V^T_k A^T ( \lambda-\mu)|_{\mathcal{B}_k} = 0,
  \end{aligned}
\end{equation}
or in matrix form 
\begin{equation} \label{dual_system}
    \begin{pmatrix}
      -\mathbb{1}^T_{B_{k,<}} & \mathbb{1}^T_{B_{k,<}} \\
      V_k^T A^T      |_{B_{k,<}} & V_k^T A^T |_{B_{k,>}}
    \end{pmatrix}
    \begin{pmatrix}
      \lambda_{\mathcal{B}_{k,<}}\\
      -\mu_{\mathcal{B}_{k,>}}
    \end{pmatrix}
    =
        B_k^T     \begin{pmatrix}
      \lambda_{\mathcal{B}_{k,<}}\\
      -\mu_{\mathcal{B}_{k,>}}
        \end{pmatrix}
        =
    \begin{pmatrix}
      -1\\
      0
    \end{pmatrix} = -c_{\mathcal{B}_k},
\end{equation}
where $c$ is the cost vector of the LP problem.

Due to the complementarity condition, for all $i \in \mathcal{N}_k$
the $\lambda_i=0$ and $\mu_i=0$.

At this point we satisfy the following KKT conditions. We are feasible
for \eqref{primal_max}, we satisfy the dual equations
\eqref{kkt_dual_first} and \eqref{kkt_dual_second} and we satisfy the
complementarity conditions.
  
If $\lambda \ge 0$ and $\mu \ge 0$, the solution is optimal for the
current Krylov subspace and $\mathcal{B}^*_k = \mathcal{B}_k$.  The
optimal guess for the solution within $\mathcal{K}_k$ is then $x_k =
x_0 + V_k y^*_k$.
  
\subsection{Updating the basic set $\mathcal{B}_k$} \label{sec:pivot infty}
However, if one of the Lagrange multipliers from \eqref{dual_system}
is negative, we can improve the solution by a pivot. Suppose the
Lagrange multiplier with index $q \in \mathcal{B}_k$ is negative.  We
then make the $q$-th constraint inactive by modifying the system as
\begin{equation}
   \begin{pmatrix}
    -\mathbb{1}|_{\mathcal{B}_{k,<}}  &AV_k |_{\mathcal{B}_{k,<}}  \\
    \mathbb{1}|_{\mathcal{B}_{k,>}}   & AV_k |_{\mathcal{B}_{k,>}}
  \end{pmatrix}
  \begin{pmatrix}
    \gamma^+_k\\
    y^+_k
  \end{pmatrix}
  = \begin{pmatrix}
    r_0 |_{\mathcal{B}_{k,<}}\\
    r_0 |_{\mathcal{B}_{k,>}}
  \end{pmatrix}
  + \alpha e_j,
\end{equation}
where $e_j$ is the $j$-th unit vector, where $j$ is the position of
$q$ in the subset $\mathcal{B}_k$.

We can write
\begin{equation}
  B_k \begin{pmatrix}
    {\gamma_k}^+\\
    y_{k}^+
  \end{pmatrix}
  = 
    B_k \begin{pmatrix}
    \gamma_k\\
    y_{k} 
    \end{pmatrix}
     + \alpha e_j,
\end{equation}
or
\begin{equation}
  \begin{pmatrix}
    {\gamma_k}^+\\
    y_{k}^+
  \end{pmatrix}
  = 
    \begin{pmatrix}
    \gamma_k\\
    y_{k} 
    \end{pmatrix}
    +  B_k^{-1} e_j \alpha
     =     \begin{pmatrix}
    \gamma_k\\
    y_{k} 
    \end{pmatrix}
    + d  \alpha,
\end{equation}
with a search direction $d \in \mathbb{R}^{k+1}$. 

The change in $\gamma_k$ is $\Delta \gamma_k = {\gamma_k}^+ - \gamma_k =
d_1 \alpha$ and should be negative, if $\gamma_k=\|r_k\|_\infty$
diminishes.  We can define $\tilde{d} = d_{2:end}/d_1 \in \mathbb{R}^k$. We then have
\begin{equation}
  \begin{aligned}
    y_{k}^+  &=    y_{k} + \tilde{d}\Delta \gamma_k, \\
    r_k^+ &= r_k - AV_k\Delta y = r_k - AV_k\tilde{d} \Delta \gamma_k.
  \end{aligned}
\end{equation}
What is the maximal decrease in $\gamma_k$ we can realize while
keeping primal feasibility?  For each choice of $\Delta\gamma_k$,
the following inequalities should remain valid
\begin{equation}
  - {\gamma_k}^+   \le (r_k - AV_k \Delta y_k )_i\le {\gamma_k}^+  \quad \forall i \in \mathcal{N}_k \cup \{q\}.
\end{equation}
We include $\{q\}$ because, the opposite bound of the removed
constraint can become active. Or
\begin{equation}
  -(\gamma_k + \Delta \gamma_k)  \le \left(r_k - AV_k \tilde{d} \Delta \gamma_k)\right)_i \le \gamma_k + \Delta \gamma_k, \quad \forall i \in \mathcal{N}_k \cup \{q\}.
\end{equation}
This leads to the following inequalities that should remain valid for
any choice of $\Delta \gamma$
\begin{eqnarray}
  (-\mathbb{1}_m + AV_k \tilde{d})_i  \Delta \gamma_k   \le (r_k)_i  + \gamma_k, \quad \forall i \in \mathcal{N}_k \cup \{q\},\\
  -(\mathbb{1}_m + AV_k \tilde{d})_i  \Delta \gamma_k   \le \gamma_k -(r_k)_i,  \quad \forall i \in \mathcal{N}_k \cup \{q\}.
\end{eqnarray}
Since we know that the right hand sides are positive or zero for the
current guess, and $\Delta \gamma$ should be negative, we only have to
check the bounds for indices where the factor of the left hand side is
negative.

The step size $\Delta \gamma_k$ is then determined by the bounds.  We
determine the index $r \in \mathcal{N}_k$ as follows:
\begin{equation}
  \Delta \gamma_{k,<}  : = \max_{ i \in \mathcal{N}_k \cup \{q\}, (-\mathbb{1}_m + AV_k \tilde{d})_i < 0}  \frac{(r_k  + \gamma_k)_i}{(-\mathbb{1}_m + AV_k \tilde{d})_i}
\end{equation}
and
\begin{equation}
   \Delta \gamma_{k,>}  : = \max_{ i \in \mathcal{N}_k \cup \{q\}, -(\mathbb{1}_m + AV_k \tilde{d})_i < 0} \frac{\gamma_k - (r_k)_i}{-(\mathbb{1}_m + AV_k \tilde{d})_i}.  
\end{equation}
We then have
\begin{equation}
  \Delta \gamma_k = \max(\Delta \gamma_{k,<}, \Delta \gamma_{k,>}).
\end{equation}
We now update the coefficients, the bound and the basic set as
\begin{equation}
  \begin{cases}
    y_k &\leftarrow y_k + \tilde{d} \Delta \gamma_k,\\
    \gamma_k &\leftarrow \gamma_k + \Delta \gamma_k,\\
    \mathcal{B}_k &\leftarrow \mathcal{B}_k \cup \{r\} \setminus \{ q\}, \\
    r_k &\leftarrow r_k - AV\tilde{d} \Delta \gamma_k.
  \end{cases}
\end{equation}
Note that updating the coefficients in this way eliminates the need to
solve the primal equation, described in \eqref{lmax_primal}, in the next
iteration.

\begin{remark}
  The method can have degenerate steps when a $|(r_0-V_ky_k)_q| =
  \gamma_k$ for a $q \notin \mathcal{B}_k$.  Then it finds a steps
  size $\Delta \gamma_k = 0$.
\end{remark}

\subsection{
  Initial basic set for $\mathcal{B}_{k+1}$ for Krylov subspace
  $\mathcal{B}_{k+1}$} \label{sec:expanding infty}

Suppose we have found the optimal basic set $\mathcal{B}^*_k$ for
Krylov subspace $\mathcal{K}_k$ with basis $V_k$.  The Krylov basis is
now expanded to $V_{k+1} = [ V_k, v_{k+1}]$ with Arnoldi or
Golub-Kahan bi-diagonalization, depending on the dimensions of the matrix
$A$.  We now determine the initial basic set $\mathcal{B}_{k+1}$ for
$\mathcal{K}_{k+1}$.

\begin{definition}
  Let $\mathcal{B}_k^*$ be the optimal basic set for Krylov subspace
  $\mathcal{K}_k$ with $V_k$ as basis.  Let $V_{k+1} = [V_k, v_k]$ be
  a basis for Krylov subspace $\mathcal{K}_{k+1}$.  An initial basic
  feasible guess $\mathcal{B}_{k+1}$ for Krylov subspace
  $\mathcal{K}_{k+1}$ is determined by the solution the following
  auxiliary problem
  \begin{equation}
    \begin{aligned}
      \min_{\gamma_{k+1}, y_{k+1}} &\,\gamma_{k+1},\\
\text{s.t.}&\,      |A(x + V_{k+1} y_{k+1})-b|_{i\in \mathcal{B}^*_k} = \gamma_{k+1},\\
 &     |A(x + V_{k+1}y_{k+1})-b|_{i\in \mathcal{N}^*_k} \le  \gamma_{k+1}.
    \end{aligned}
  \end{equation}
\end{definition}

For this auxiliary problem it easy to create a feasible
solution. Indeed, a solution $(\gamma^*_k,y^*_k)$ of $\mathcal{K}_k$
is a feasible solution for the auxiliary problem with $y_{k+1} = (y^*_k,
0)$. However, it only has $k+1$ active constraints, while the solution
of the auxiliary problem can have $k+2$ active constraints, the $k+1$
existing constraints from the $\mathcal{B}^*_k$ and one additional
constraint from the $\mathcal{N}_k$ that becomes active in the optimal
point.  This gives a initial feasible basic set
$\mathcal{B}_{k+1}$.

We write $y_{k+1} = (y^+_k,\alpha)$ and use an initial guess $(y^*_k,0)$,
the solution of the previous Krylov problem. We then change $\alpha$,
keeping the constraints $\mathcal{B}^*_k$ active, until we make one
additional constraint active.

So we change $\gamma^+_{k+1}$, $y_k^+ \in \mathbb{R}^k$ and $\alpha$,
i.e. $k+2$ variables, while they satisfy the $k+1$ equality
constraints
\begin{equation}\label{one-dimensional}
  \begin{aligned}
     &\, \begin{pmatrix}
        -\mathbb{1}_{\mathcal{B}_{k,<}}  & AV_k |_{\mathcal{B}_{k,<}} &Av_{k+1} |_{\mathcal{B}_{k,<}}\\
        \mathbb{1}_{\mathcal{B}_{k,>}}  & AV_k |_{\mathcal{B}_{k,>}}  &Av_{k+1}|_{\mathcal{B}_{k,>}}
      \end{pmatrix}
      \begin{pmatrix}{\gamma_{k+1}}^+\\ y^+_k \\ \alpha
      \end{pmatrix}
      = \begin{pmatrix}
        (r_0) |_{i \in \mathcal{B}_{k,<}}\\
        (r_0) |_{i \in \mathcal{B}_{k,>}}
      \end{pmatrix}.
  \end{aligned}
\end{equation}
This defines a one-dimensional subspace
\begin{equation}
  \begin{pmatrix}
    -\mathbb{1}_{\mathcal{B}_{k,<}}  & AV_k |_{\mathcal{B}_{k,<}} \\
    \mathbb{1}_{\mathcal{B}_{k,>}}  & AV_k |_{\mathcal{B}_{k,>}}  
  \end{pmatrix}
  \begin{pmatrix}\Delta \gamma_{k}\\ \Delta y_k 
  \end{pmatrix}
  =
  -
\begin{pmatrix}
Av_{k+1} |_{\mathcal{B}_{k,<}}\\
Av_{k+1}|_{\mathcal{B}_{k,>}}
\end{pmatrix}
 \alpha.
\end{equation}
We introduce a search direction $d \in \mathbb{R}^{k+1}$ that
satisfies $B_k d = - Av_{k+1} |_{\mathcal{B}_{k}}$ and write
\begin{equation}
 \begin{pmatrix}
    {\gamma_{k}}^+\\
    y_k^+\\
    \alpha 
  \end{pmatrix}
  = \begin{pmatrix}
    \gamma_k\\
    y^*_k\\
    0 
  \end{pmatrix}
  +
  \begin{pmatrix}
    \Delta \gamma_k\\
    \Delta y_k\\
    \alpha
  \end{pmatrix}
  = 
  \begin{pmatrix}
    \gamma_k\\
    y^*_k\\
    0 
  \end{pmatrix}
  +
  \begin{pmatrix}
    d\\
    1
  \end{pmatrix}
  \alpha 
\end{equation}
The bound $\gamma_k$ should decrease to $\gamma_{k+1}$. We have
$\Delta \gamma_k = \gamma_{k+1} - \gamma_k = d_1 \alpha$, which should be
negative. We set $\alpha = \Delta \gamma_k / d_1$ and write
\begin{equation}
  \begin{pmatrix}
    y_{k}^+ \\
    \alpha
  \end{pmatrix}
  =
  \begin{pmatrix}
    y^*_{k}\\0
  \end{pmatrix}
  +  \begin{pmatrix}
    d_{2:k+1}/d_1 \\
    1/d_1
  \end{pmatrix}
  \Delta \gamma_k
 = 
 y^{(0)}_{k+1}
  + \tilde{d} 
  \Delta \gamma_k,
\end{equation}
where $\Delta \gamma_{k}$ is now the step size.  The change in the
direction of the new vector in the Krylov subspace, $v_{k+1}$, is now
$\Delta \gamma_{k}/d_1$.

When we diminish $\gamma_k$ both $y_k^+$ and $\alpha$ change
accordingly. At some point it will hit the boundary of one of the
other constraints in $\mathcal{N}_k$. The index of this constraint is
now added to the active set.  This increases the active set from $k$
to $k+1$ active constraints and gives a basic feasible set $B_{k+1}$.

The maximal step size along the search direction $\tilde{d}$ is determined by
the inactive constraints.  We have
\begin{equation}
  -{\gamma_k}^+ \le (r^*_k - AV_{k+1} \tilde{d} \Delta \gamma_k ) )_i \le {\gamma_k}^+   \quad i \in \mathcal{N}_k.
\end{equation}
This leads to the inequalities
\begin{equation}
  \begin{aligned}
  \left(-\mathbb{1}_m + AV_{k+1} \tilde{d} \right) \Delta \gamma_k  \le r^*_k + \gamma_k \mathbb{1}_m, \\
  \left(-\mathbb{1}_m - AV_{k+1} \tilde{d} \right) \Delta \gamma_k \le  \gamma_k\mathbb{1}_m - r^*_k.
  \end{aligned}
\end{equation}
Again, the right hand sides are positive.  Since $\Delta \gamma_k$ is
negative, we need to check the inequalities only where the factors on
the left are negative. This gives
\begin{equation}
  \begin{aligned}
    \Delta \gamma_{k,<}  := \max_{ i \in \mathcal{N}_k, (-\mathbb{1}_m + AV_{k+1} \tilde{d})_i < 0} \frac{ (r_k^*)_i + \gamma_k}{(-\mathbb{1}_m + AV_{k+1} \tilde{d})_i}, \\
    \Delta \gamma_{k,>} := \max_{ i \in \mathcal{N}_k, -(\mathbb{1}_m + AV_{k+1} \tilde{d})_i < 0} \frac{ \gamma_k - \left( r_k^* \right)_i}{- (\mathbb{1}_m + AV_{k+1} \tilde{d})_i},
  \end{aligned}
\end{equation}
the step size is then $\Delta \gamma_k = \max( \Delta \gamma_{k,<}, \Delta
\gamma_{k,>} )$. The index $r$ is where the bound is reached.  The initial
guess for the expanded Krylov subspace $\mathcal{K}_{k+1}$ is then
\begin{equation}
  \begin{cases}
    y_{k+1} &= \begin{pmatrix} y^*_k \\ 0 \end{pmatrix}  + \begin{pmatrix} d \\ 1 \end{pmatrix}  \Delta \gamma_k /d_1,\\
    \gamma_k &= \gamma^*_k + \Delta \gamma_k,\\
    \mathcal{B}_{k+1} &= \mathcal{B}^*_k \cup \{r\}, \\
    r_{k+1}  &= r_k - AV_{k+1}\tilde{d} \Delta \gamma_k.
  \end{cases}
\end{equation}
This results in algorithm~\ref{algo:l_infty} that is discussed in detail in section~\ref{sec:numerical results}. 

\begin{remark}
  The change in convergence is now
  \begin{equation}
    \begin{aligned}
      \|r_{k+1}\|_\infty-\|r_k\|_\infty &= \Delta \gamma_k = d_1 =
      \begin{pmatrix}
        1 & 0 & \ldots &0
      \end{pmatrix}^T d = -c^T_{\mathcal{B}_k} B_k^{-1} (Av_{k+1})|_{\mathcal{B}_k} \\
      & = (\lambda-\mu)_{\mathcal{B}_k}^T (Av_{k+1})|_{\mathcal{B}_k} = (\lambda-\mu)^T A v_{k+1},
    \end{aligned}
  \end{equation}
  where we use that $B_k^T (\lambda- \mu)|_{\mathcal{B}_k} = -c_{\mathcal{B}_k}$, from \eqref{dual_system}.   
  So the Krylov subspace expansion will only lead to an improvement
  when ${(\lambda-\mu)^T A v_{k+1} < 0}$. 
\end{remark}

\subsection{Updating the $QR$ factorization of $B_k$ and $B_k^T$ after rank-one updates}
The matrices $B_k$ and $B_k^T$ that need to be solved in both
algorithms are small and dense.  Furthermore, in the sequence they
appear in the algorithm they only differ by a rank-one
updates.  Hence we can reuse the factorization.  Reusing the
factorizations in the simplex is commonplace and essential for
performance. There is extensive literature on updating the $L$ and $U$
of a sparse LP problem \cite{elble2012review}.

Here, however, we are dealing with small dense matrices where we can use QR updates\cite{golub13}.  Rank-one updates to a matrix $B$ can be written as
\begin{equation}
  B^\prime = B + u v^T
\end{equation}
where $u$ and $v$ are vectors. 

Suppose we already have a $QR$ factorization of $B$ and we are now
interested in efficient calculation of $Q^\prime$ and $R^\prime$.
\begin{equation}
  Q^\prime R^\prime = QR + u v^T
\end{equation}

To make indexing easier in the implementation of $QR$ updates, we do
an in-place replacement when we update $\mathcal{B}_k = \mathcal{B}_k
\cup \{ r \} \setminus \{ q\}$. It means that $r$ gets the same position as
$q$ in the set $\mathcal{B}_k$.

As a result, the indices that indicate active lower or upper bounds
appear unordered in the $\mathcal{B}_k$. We then write
\begin{equation}
  B_k = \begin{pmatrix}
    (\pm 1) |_{\mathcal{B}_k}  & AV_k
  \end{pmatrix}
\end{equation}
where a $-1$ appear in the first column if the first element
$\mathcal{B}_k$ is a lower bound or a $+1$ of it is an upper bound.  We
then have the following update
\begin{equation} \label{qr_update_basic_max}
  B_k := 
  \begin{pmatrix}
    (\pm 1)|_{\mathcal{B}_k}  &AV_k |_{\mathcal{B}_k} 
  \end{pmatrix}
  +   e_j \left[ \begin{pmatrix}
    (\pm 1)_r  &(AV_k)_r
  \end{pmatrix}
  -
  \begin{pmatrix}
    (\pm 1)_q  &(AV_k)_q
  \end{pmatrix}
  \right].
\end{equation}
where $j$ is the position of $q$ in the set $\mathcal{B}_j$.

When we expand the Krylov subspace, we have
\begin{eqnarray}\label{qr_update_krylov_max}
  \begin{pmatrix}
    \pm \mathbb{1}|_{\mathcal{B}_{k+1}}  &AV_{k+1} |_{\mathcal{B}_{k+1}}  
  \end{pmatrix}
   = 
  \begin{pmatrix}
    \pm \mathbb{1}|_{\mathcal{B}_{k}}  &AV_k |_{\mathcal{B}_{k}}  & 0  \\
    0 & 0 &0 
  \end{pmatrix} \nonumber \\
  +
  \begin{pmatrix}
    0   & 0  & 0  \\
    \pm \mathbb{1}|_{r}   & AV_k |_{r}  & Av_{k+1}|_r
  \end{pmatrix}
    +
    \begin{pmatrix}
    0  &  0 &Av_{k+1} |_{\mathcal{B}_{k}}  \\
    0 & 0 &0 
    \end{pmatrix}.
\end{eqnarray}

\begin{algorithm}
  \begin{algorithmic}[1]
    \STATE{$r_0 = b-Ax_0$}
    \STATE{$\gamma_0, i = \max_i |(r_0)_i|$ }
    \STATE{ $\mathcal{B}_0 = \{i\}$, index where the max is reached.} \;
    \STATE{$V_1 = \left[ r_0/\|r_0\| \right]$ }
    \FOR{$k=1,\ldots, $maxit}
    \STATE{ Calculate $AV_k = [AV_{k-1} Av_k]$ and store and update $V_{k+1} =[V_k, v_{k+1}]$.}
    \IF{$k=1$}
       \STATE{$B_1 = \begin{pmatrix} \pm 1  \end{pmatrix}$}, \;
       \STATE{[$Q,R$] = qr($B_1$)}. \;
    \ENDIF
    \STATE { Solve $B_k d =-  Av_{k+1}|_{\mathcal{B}_k} $ using $Q$ and $R$.}\;
    \STATE { Set $d := [d(2:end)/d1 ; 1/d(1)]$}
    \STATE { Set $y_k = [y_k ; 0]$} 
    \STATE{ $r_k$, $\Delta \gamma$, $y_k$, $r$ = StepSizeKrylovExpansion (). }\;
    \STATE{   $\mathcal{B}_k = \mathcal{B}_{k-1} \cup \{r\}$, } \;
    \WHILE{ $l=1, \ldots, $ maxsimplexiters }
    \IF{$l=1$}
    \STATE{Update the $Q$ and $R$ using \eqref{qr_update_krylov_max}.}\;
    \ELSE{}
    \STATE{Update the $Q$ and $R$ using \eqref{qr_update_basic_max}.}\;
    \ENDIF
      \STATE{
        Solve $B_{k}^T \begin{pmatrix}\lambda_{\mathcal{B}_<}  \\ -\mu_{\mathcal{B}_>} \end{pmatrix} =\begin{pmatrix}  
      1\\
      0
      \end{pmatrix}$  using $Q$ and $R$.}\;
      \IF{ $\lambda \ge 0$ and $\mu \ge 0$}
      \STATE{        Break; Solution Found. }
      \ELSE{}
        \STATE {$q = \min(\lambda_i, \mu_i)$  leaving index,}
        \STATE{ $\mathcal{B}_k = \mathcal{B}_k \setminus \{q\}$, }
        \STATE{Solve $B_k d = e_q$ using $Q$ and $R$, }
        \STATE{Set $d:=d(2:end)$, Set $d_1 = d(1)$,}
        \STATE{ $r_k$, $\Delta \gamma$, $y_k$, $r$ = StepSizeSimplex (), }
        \STATE{ $\mathcal{B}_k  = \mathcal{B}_l \cup \{r \}$. }
        \ENDIF
        \ENDWHILE      
        \STATE{ $x_k = x_0 + V_k y_k $. }
        \STATE{$\|r_k\|_\infty = \gamma_k$. }
        \ENDFOR
  \end{algorithmic}
 \caption{Krylov-Simplex for $ \min_{x \in x_0 +\mathcal{K}_k }\|b-Ax \|_\infty$ \label{algo:l_infty}}
\end{algorithm}

\begin{algorithm}
  \begin{algorithmic}[1]
    \STATE{\textbf{StepSizeKrylovExpansion}() }
    \STATE{calculate $AV_kd$ and store}
    \STATE{$\Delta \gamma_< =\max_{ i \in \mathcal{N}, \cdot <0} \frac{(r_k + \gamma \mathbb{1})_i}{(-\mathbb{1}+ AV_kd)_i}$}
    \STATE{$\Delta \gamma_> = \max_{ i \in \mathcal{N}, \cdot <0} \frac{(r_k - \gamma \mathbb{1})_i}{(\mathbb{1}+ AV_kd)_i}$}
    \STATE{$\Delta \gamma = \max(\Delta \gamma_<, \Delta \gamma_>)$}
    \STATE{$r$ is blocking index}
    \STATE{$y_k^+ = y_k + d \Delta \gamma$}
    \STATE{$r_k^+ = y_k - AV_kd \Delta \gamma$}
  \end{algorithmic}
  \caption{Blocking function Krylov}
\end{algorithm}

\begin{algorithm}
  \begin{algorithmic}[1]
    \STATE{\textbf{StepSizeSimplex}() }
    \STATE{remove $q$ from $\mathcal{N}$}
    \STATE{calculate $AV_kd$ and store}
    \STATE{$\Delta \gamma_< =\max_{ i \in \mathcal{N}, \cdot <0} \frac{(r_k + \gamma \mathbb{1})_i}{(-\mathbb{1}+ AV_kd/d_1)_i}$}
    \STATE{$\Delta \gamma_> = \max_{ i \in \mathcal{N}, \cdot <0} \frac{(r_k - \gamma \mathbb{1})_i}{(\mathbb{1}+ AV_kd/d_1)_i}$}
    \STATE{$\Delta \gamma = \max(\Delta \gamma_<, \Delta \gamma_>)$}
    \STATE{$r$ is blocking index}
    \STATE{$y_k^+ = y_k + d \Delta \gamma /d_1$}
    \STATE{$r_k^+ = y_k - AV_kd \Delta \gamma /d_1$}
  \end{algorithmic}
  \caption{Step Size Simplex \label{stepsizesimplex}}
\end{algorithm}


\section{Krylov-Simplex for the $\ell_1$-norm} \label{sec:l1}
In this section we derive a Krylov-simplex method for the minimization
of $r_k = b-Ax_k$ in the $\ell_1$-norm with ${x_k \in x_0 +
\mathcal{K}_k}$.  This leads to  a similar linear programming problem
as in the previous section, where we used the $\ell_\infty$-norm.

Subsection \ref{sec:lp 1} formulates the LP problem and the
corresponding KKT conditions.  In section \ref{sec: primal dual 1}, we
solve the primal and dual equations for a given $\mathcal{B}_k$.  The
next section, \ref{sec:pivot 1}, we discuss how we update the basic
set $\mathcal{B}_k$. The section \ref{sec: update 1} discusses the
initial $\mathcal{B}_{k+1}$ for Krylov $\mathcal{K}_{k+1}$.

\begin{definition}\label{one-norm Krylov}  Let $A \in \mathbb{R}^{m \times n}$ a matrix
  and $b \in \mathbb{R}^m$ a right hand side.  The iterates of the
  $\ell_1$-norm Krylov-Simplex are given by
  \begin{equation}
    x_k := \text{argmin}_{ x \in x_0 + \mathcal{K}_k} \|r_k\|_1.
  \end{equation}
\end{definition}


\subsection{Optimization problem} \label{sec:lp 1}
Let $V_k$ be a basis for $\mathcal{K}_k$ and let the guess for the
solution be $x_k = x_0 + V_k y_k$ with coefficients $y_k$.  These
coefficients are determined by the following projected problem
\begin{equation}
  \begin{aligned}
    \min &\, \mathbb{1}_m^T \gamma_k ,\\
    \text{s.t}& \, -\gamma_k \le r_0-AV_k y_k \le \gamma_k.
  \end{aligned}
\end{equation}
where $\gamma_k$ is now a vector in $\mathbb{R}^m$.
The corresponding Lagrangian is
\begin{equation}
  \mathcal{L}(\gamma_k,y_k,\lambda,\mu) = \mathbb{1}_m^T \gamma_k-  \lambda^T \left(\left(r_0-AV_ky_k\right) +  \gamma_k\right)
  - \mu^T \left(\gamma_k - \left(r_0-AV_ky_k\right)\right),
\end{equation}
where $\lambda \in \mathbb{R}^m$ are the Lagrange multipliers for the
lower bounds and $\mu \in \mathbb{R}^m$ the Lagrange multipliers for
the upper bounds.  This leads to the KKT conditions with $i \in \{1, \ldots, m\}$
\begin{eqnarray}
\lambda + \mu &= 1, \label{firstkkt}\\
V_k^T A^T  (\lambda -  \mu) &=& 0, \label{dual}\\
 -\gamma_k  \le  (r_0-AV_k y_k) &\le& \gamma_k, \\     
   \mu \ge  0, \quad  \lambda &\ge& 0,\\
   \lambda_i ((r_0-AV_k y_k)  + \gamma_k)_i &=& 0  , \quad \forall i \in \{1,\ldots, m\}  \label{lastbutoneKKT} \\
   \mu_i (\gamma_k-(r_0-AV_k y_k))_i &=& 0,  \quad \forall i \in \{1,\ldots, m\} \label{lastkkt}
\end{eqnarray}
The first two equations, \eqref{firstkkt} and \eqref{dual}, are the
dual equations. The third is the primal equation. The last two, \eqref{lastbutoneKKT} and \eqref{lastkkt},
are the complementarity conditions.  Note the similarity with the KKT
conditions \eqref{kkt_dual_first}-\eqref{complementarity_second} in
section~\ref{sec:kkt infty}.

\begin{remark}
  Again, we note that in the solution of the KKT condition holds $\lambda-
  \mu \perp \mathcal{K}_{k+1}$ if we are using Arnoldi. Similar
  relation can be derived for Golub-Kahan.
\end{remark}

\subsection{Solving the primal and dual equations from the KKT system for Krylov subspace $\mathcal{K}_k$} \label{sec: primal dual 1}

\begin{definition}
  The basic set $\mathcal{B}_k \subset \{1,\ldots,m\}$ for Krylov 
  subspace $\mathcal{K}_k$ for the $\ell_1$-norm is a subset with $k$ indices where
  \begin{equation}
    (r_0 - AVy_k)_i = (\gamma_k)_i=0 \quad \forall i \in \mathcal{B}_k
  \end{equation}
  and
  \begin{equation}
    B_k = (AV_k)|_{\mathcal{B}_k}
  \end{equation}
  is a non-singular $k$ by $k$ matrix.
\end{definition}

 The non-basic set is then $\mathcal{N}_k = \{ 1,\ldots, m\}\setminus
 \mathcal{B}_k$ and $|\mathcal{N}_k|=m-k$.

For a given basic set $\mathcal{B}_k$, the coefficients $y_k \in
\mathbb{R}^k$ are then determined from the system
\begin{equation} \label{eq:coefficient system}
  (AV_k)|_{\mathcal{B}_k} y_k = B_k y_k= r_0|_{\mathcal{B}_k}. 
\end{equation}
The solution leads to the guess ${x_k = x_0 + V_ky_k}$.

We now determine the remaining $\gamma_i$, for $i \in \mathcal{N}_k$, as
\begin{equation}
  \gamma_i = |r_0 -AV_k y_k|.
\end{equation}
We can now split $\mathcal{N}_k = \mathcal{N}_{k,<} \cup \mathcal{N}_{k,>}$, 
where $\mathcal{N}_{k,<}$ are the indices where ${\gamma_i=-(r_0-AV_ky_k)_i}$ and
$\mathcal{N}_{k,>}$ are the indices where ${\gamma_i=(r_0-AV_ky_k)_i}$.

If $i \in \mathcal{N}_{k,<}$, then ${\gamma_i + (r_0 -AV_k y_k)_i = 0}$
while ${\gamma_i - (r_0 -AV_k y_k)_i \neq 0}$. So from the last
complementarity condition in the KKT conditions
\eqref{firstkkt}-\eqref{lastkkt} we find that the Lagrange multiplier $\mu_i=0$ for all ${i \in
\mathcal{N}_{k,<}}$ and, from the first equation from the KKT, we find
that $\lambda_i =1$.

Similarly, we find that $\mu_i=1$ for all $i \in \mathcal{N}_{k,>}$.
So we have that
\begin{equation}\label{nonbasic_lambda}
  (\lambda_i,\mu_i)
  =\begin{cases}
  (0,1)  \quad \text{if} \quad i \in \mathcal{N}_{k,>}, \\
  (1,0)  \quad \text{if} \quad i \in \mathcal{N}_{k,<}.
  \end{cases}
\end{equation}
The remaining $\lambda_i$, $\mu_i$, with $i \in \mathcal{B}_k$, are
found from equation \eqref{dual}, i.e.
\begin{equation}
  V_k^T A^T  (\lambda -  \mu) = 0, 
\end{equation}
or, after padding $\lambda_\mathcal{N}$, and $\mu_\mathcal{N}$ with
zeros to create vectors of length $m$, we find
\begin{equation} \label{l1 dual equation}
  V_k^T A^T ( \lambda-\mu)|_{\mathcal{B}_k}= B_k^T z  = - \left( V_k^TA^T\left( \lambda_{\mathcal{N}_{k,<}}-\mu_{\mathcal{N}_{k,>}} \right)\right)|_{\mathcal{B}_k}.
\end{equation}
The right hand side is known from \eqref{nonbasic_lambda} and we solve
the $k$ by $k$ system for $z = (\lambda-\mu)|_{\mathcal{B}_k}$.  We
  then have
\begin{equation}  \label{l1_lagrange}
  \begin{cases}
    \lambda_i + \mu_i  &= 1,       \\
    \lambda_i - \mu_i  &= z_i,
  \end{cases}
  \quad  \forall i \in \mathcal{B}_k
\quad \text{and} \quad
  \begin{cases}
    \lambda_i = (1+ z_i)/2,\\
    \mu_i = (1 - z_i)/2,
  \end{cases}
  \quad  \forall i \in \mathcal{B}_k.
\end{equation}
We now have the full vector $\lambda$ and $\mu$.

At this point we have feasible solution of the primal equation, we
have solved the dual equations, \eqref{firstkkt} and \eqref{dual}, and
we satisfy the complementarity conditions.

If $\lambda \ge 0$ and $\mu \ge 0$, we have found the optimal basic
set $\mathcal{B}^*_k$ for the current Krylov subspace $\mathcal{K}_k$.
The optimum for the current subspace is now $x_k = x_0 + V_k y^*_k$.


\subsection{Updating the basic set $\mathcal{B}_k$} \label{sec:pivot 1}
However, if one of the Lagrange multipliers in $\mathcal{B}_k$, from
\eqref{l1_lagrange}, is negative, we can improve the objective by
updating the basic set $\mathcal{B}_k$.  We remove the index $q$,
corresponding to the negative Lagrange multiplier, from
$\mathcal{B}_k$ and look for an entering index in $\mathcal{N}_k$.

We modify the linear system, \eqref{eq:coefficient system},  as follows
\begin{equation}
  A V_k |_{\mathcal{B}_k} y_k^+ = r_0 |_{\mathcal{B}_k}  + e_j \alpha.
\end{equation}
where $j$ is the position of  $q$ in the set $\mathcal{B}_k$.

Depending whether $\lambda_q <0$ or $\mu_q<0$ is negative, we perturb
the system with a positive or negative $\alpha$.  Indeed, when
$\lambda_q$ is negative that means that the lower bound should become
inactive to improve the objective. This means that $-\Delta \gamma_q < (r_0-
AV_ky_k^+)_q$ while the upper bound remains $(r_0- AV_ky_k^+)_q =
\Delta \gamma_q$.  This means that $(AV_k y_k^+)_q = (r_0)_q + e_q \alpha$ with
$\alpha<0$.

Similarly, if $\mu_q$ is negative, the upper bound should become
inactive to improve the objective. This is $(r_0- AV_ky_k^+)_q <
(\Delta \gamma_k)_q$ while $-(\Delta \gamma_k)_q = (r_0- AV_ky_k^+)_q$.  In other words
$(AV_k y_k^+)_q = (r_0)_q + e_q \alpha$ with $\alpha>0$.

We write $y_k^+ = y_k + d \alpha$, where the search direction $d$ is
solution of
\begin{equation}
  A V_k |_{\mathcal{B}_k} d = \begin{cases}
      -e_j \quad  \text{if} \quad \lambda_q < 0,\\
      e_j \quad  \text{if} \quad \mu_q < 0,
  \end{cases}
\end{equation}
with $e_j$ is the $j$-th unit vector, where $j$ is the position of $q
\in \mathcal{B}_k$.  The sign switch is introduced to keep $\alpha \ge
0$.

We increase now the step size $\alpha$ until an index $r \in \mathcal{N}_k$
becomes active.  The step size is determined from $(r_0 - AV_k(y_k + d \alpha))_i =0$
\begin{equation}
  \alpha = \min_{ i \in \mathcal{N}_k, \cdot > 0} \frac{(r_k)_i}{(AV_k d)_i} .
\end{equation}
The index $r \in \mathcal{N}_k$ is the index where this minimum is
reached.

We now update
\begin{equation}
  \begin{cases}
    y^+_k  &\leftarrow y_k + d \alpha, \\
    \mathcal{B}_k &\leftarrow \mathcal{B}_k  \setminus \{q\} \cup \{r\},\\
    r^+_k &\leftarrow r_k - AV_k  d\alpha.
  \end{cases}
\end{equation}

\begin{remark}
  Note that by choosing your step size $\alpha$ this way we have that
  $N_<$ and $N_>$ remains unchanged.
\end{remark}

\begin{remark}
We want to avoid that $V^T A^T (\lambda -\mu)_{\mathcal{N}_k}$, right hand side of \eqref{l1 dual equation},  is calculated each
iteration. So we want to efficiently update this right hand side. 

The algorithm first decides which $q$  is removed from $\mathcal{B}_k$ based on the
Lagrange multipliers.  Depending on the Lagrange multiplier it will be
added to $\mathcal{N}_<$ or $\mathcal{N}_>$

For example, if $\mu_q \le 0 $. We will the make the upper bound
inactive. As a result $\mu_q=0$ and $\lambda_q = 1$.  We then have
that $\mathcal{B}^+ = \mathcal{B}_k \setminus \{q \} $ and
$\mathcal{N}^+_< = \mathcal{N}_< \cup \{q\}$.  On the other hand, if
$\lambda_q \le 0$ we then make the lower bound inactive. As a result
$\mu_q=1$ and $\lambda_q = 0$ and $\mathcal{B}^+ = \mathcal{B}_k
\setminus \{q \}$ and $\mathcal{N}_> = \mathcal{N}_>\cup \{q\}$

In addition, the step size $\alpha$ determines $r$, the entering index.
This also has an influence on the $\lambda$ and $\mu$ for
$\mathcal{N}$ Since it is removed from $\mathcal{N}_<$ or
$\mathcal{N}_>$?  We now have for four cases
 \begin{equation}
    \lambda^+ - \mu^+ = \lambda -\mu
    +
    \left \{
    \begin{aligned}
      (1_q - 1_r) - (0)  &=  1_q-1_r  &&   \mathcal{N}_<^+ = \mathcal{N}_< \cup \{q\} \setminus \{r\} \, \text{and}\,\,\,  \mathcal{N}^+_< =  \mathcal{N}_>  \\
      (1_q   ) - (-1_r) &=  1_q+1_r &&  \mathcal{N}_<^+ = \mathcal{N}_< \cup \{q\}       \,\text{and} \,            \mathcal{N}_<^+ = \mathcal{N}_>  \setminus \{r\} \\
      (-1_r) -  (1_q)    &=  -1_q-1_r   &&  \mathcal{N}_<^+ =\mathcal{N}_< \setminus \{r\}  \,\text{and}\,\,\,  \mathcal{N}_>^+ =\mathcal{N}_> \cup \{q\} \\
      (0)   -(1_q-1_r)  &=  1_r-1_q   &&  \mathcal{N}_<^+ = \mathcal{N}_<   \,   \text{and}\,\,\,  \mathcal{N}_>^+ = \mathcal{N}_> \cup \{q\} \setminus \{r\}
    \end{aligned} \right.
 \end{equation}
With the same reasoning we can update
  \begin{equation} \label{update_rhs}
  V_k^T A^T \left( \lambda_{\mathcal{N}^+} -\mu_{\mathcal{N}^+} \right)
  = V_k^T A^T \left( \lambda_{\mathcal{N}} -\mu_{\mathcal{N}} \right)
  +  V_k^T A^T \left( (\lambda^+-\lambda)-(\mu^+ -\mu) \right).
  \end{equation}
\end{remark}
\subsection{Initial basic set $\mathcal{B}_{k+1}$ for Krylov subspace $\mathcal{K}_{k+1}$} \label{sec: update 1}
We expand the Krylov subspace $\mathcal{K}_k$ to
$\mathcal{K}_{k+1}$. The basis is now $V_{k+1} = [V_k,v_{k+1}]$.  The
question is what is the initial basic set $\mathcal{B}_{k+1}$ for
$\mathcal{K}_{k+1}$?

\begin{definition}
  Let $\mathcal{B}^*_k$ the optimal basic set for Krylov subspace
  $\mathcal{K}_k$, then the following auxiliary problem gives an
  initial basic set $\mathcal{B}_{k+1}$ for $\mathcal{K}_{k+1}$.
  \begin{equation}
      \begin{aligned}
        \min & \, \mathbb{1}_m^T \gamma_k \\
        \text{s.t.} &\, (AV_{k+1} y_{k+1})_{\mathcal{B}^*_{k}} = r_0 |_{\mathcal{B}^*_k},   \\
        &\, - \gamma_k \le r_0- (AV_{k+1} y_{k+1}) \le      \gamma_k. 
      \end{aligned}
    \end{equation}
\end{definition}

We can write
\begin{equation}
    AV_{k+1}y_{k+1} = AV_k (y_k^* + \Delta y_k) + Av_{k+1}\alpha,
\end{equation}
where $y^*_k$ is the solution of
$AV_k|_{\mathcal{B}^*_k} y^*_k = r_0|_{\mathcal{B}^*_k}$. We then
have, for a feasible solution of the auxiliary problem,
\begin{equation}
  AV_k|_{\mathcal{B}^*_k} \Delta y_k  + Av_{k+1}|_{\mathcal{B}^*_k} \alpha =0,
\end{equation}
or, with a search direction $d$ that fits
\begin{equation}
  AV_k|_{\mathcal{B}_k} d = - Av_{k+1}|_{\mathcal{B}_k},
\end{equation}
we can write $\Delta y_k = d \, \alpha$.  We now have to choose an
$\alpha$ that solves
\begin{equation}
  \begin{aligned}
    \min&\, \mathbb{1}_\mathcal{N}^T (\gamma_k)_{\mathcal{N}^*}   \\
     &\, - (\gamma_k)_i  \le \left( r_k- \left(AV_k d  + Av_{k+1} \right)\right)_i \alpha  \le (\gamma_k)_i \quad \forall i \in \mathcal{N}^*_k.
  \end{aligned}
\end{equation}
This is equivalent to
\begin{equation}
  \min_\alpha  |v - w \alpha |, 
\end{equation}
with ${v:= (r_k)_{\mathcal{N}_k^*}}$ and
${w = (AV_k d+ Av_{k+1})_{\mathcal{N}_k^*}}$, both vectors in
$\mathbb{R}^{m-k}$.

The derivative of $|v-w\alpha|$ is
\begin{equation}
  \frac{ d}{d \alpha}  \sum_{ i \in \mathcal{N}_k^*} \sqrt{(v - w \alpha)_i^2} = \sum_{ i \in \mathcal{N}_k^*}  \frac{1}{2} \frac{1}{\sqrt{ (v - w \alpha)^2_i}} 2 (v -w\alpha)_i w_i.
\end{equation}
Around $\alpha = 0$, we have the approximation
\begin{equation}
  \sum_{i \in \mathcal{N}_k^*} |(v - w \alpha)_i | \approx   \sum_{i \in \mathcal{N}_k^*} |v_i| -  \sum_{i \in \mathcal{N}_k^*} \frac{v_i}{|v_i|} w_i  \alpha + \mathcal{O}(\alpha^2).
\end{equation}
We can now select a descent direction for the auxiliary problem. It is
the direction where
$\text{sign}(r_k)_{\mathcal{N}^*_k}^T (AV_k d +
Av_{k+1})_{\mathcal{N}_k^*} > 0$, the objective of the auxiliary
problem diminishes.

We now determine the step size in the descent direction
\begin{equation}
  \alpha = \begin{cases}
    \min_{i \in \mathcal{N}^*_k, \cdot >0}   \frac{ (r_k)_i}{(AV_k d + Av_{k+1})_i} \quad &\text{if} \quad
    \text{sign}(r_k)_{\mathcal{N}^*_k}^T (AV_k d + Av_{k+1})_{\mathcal{N}_k^*} > 0,\\
    \max_{i \in \mathcal{N}_k^*, \cdot < 0} \frac{(r_k)_i}{(AV_k d + Av_{k+1})_i} \quad &\text{if} \quad
    \text{sign}(r_k)_{\mathcal{N}^*_k}^T (AV_k d + Av_{k+1})_{\mathcal{N}_k^*} < 0.
  \end{cases}
\end{equation}
The index $r \in \mathcal{N}_k$ where the minimum is reached is now
added to the active set. The initial solution for the Krylov subspace
$\mathcal{K}_{k+1}$ is now
\begin{equation}
  \begin{cases}
    y_{k+1} = \begin{pmatrix}
      y^*_k\\
      0
    \end{pmatrix}
    + \begin{pmatrix}
      d\\
      1
    \end{pmatrix}
    \alpha, \\
    \mathcal{B}_{k+1} = \mathcal{B}_k^* \cup \{r\},\\
    r_{k+1} = r_k -(AV_k + Av_{k+1})\alpha.
  \end{cases}
\end{equation}

\subsection{Updating the factorization in Krylov-simplex for  $\ell_1$}
In the Krylov-simplex algorithm for $\ell_1$ the basic matrix has
updates of a row and, when the Krylov subspace is expanded, a row and
a column.

We first look at updating a row in $AV_k|_{\mathcal{B}_k}$ when the
active set is updated as $\mathcal{B}_k = \mathcal{B}_k \cup \{q \}
\setminus \{ r\}$. We update the basic matrix as
\begin{equation}\label{qr_update_basic_l1}
  AV_k |_{\mathcal{B}_k \cup
\{q \} \setminus \{ r\} } = AV_k|_{\mathcal{B}_k}  + e_r (AV_k|_q-AV_k|_r)^T.
\end{equation}

When the Krylov subspace expands we extend  the basic set as
$\mathcal{B}_{k+1} = \mathcal{B}_k \cup \{r\}$.  We already have a QR
factorization of $AV_k|_{\mathcal{B}_k}$.  We can now write the matrix
$A_kV |_{\mathcal{B}_{k+1}}$ as two rank-one updates
\begin{equation}\label{qr_update_krylov_l1}
  AV_k |_{\mathcal{B}_{k+1}} =  \begin{pmatrix}
    AV_k |_{\mathcal{B}_k} & Av_{k+1}|_{\mathcal{B}_k}\\
    AV_k |_r       & Av_{k+1}|_r
  \end{pmatrix}
  = \begin{pmatrix}
    QR & 0 \\
    0 & 0 
  \end{pmatrix}
  + \begin{pmatrix} 0 & 0 \\ AV |_r & Av_{k+1}|_r
  \end{pmatrix}
  +
  \begin{pmatrix}
    0 & Av_{k+1}|_{\mathcal{B}_k}\\
    0 & 0 
  \end{pmatrix}.
\end{equation}
  
\begin{algorithm}
  \begin{algorithmic}[1]
    \STATE{ $r_0 = b-Ax_0$. }
    \STATE{ $\|r_0\|_1 = \sum |(r_0)_i|$.}
    \STATE {$\mathcal{B}_1 = \text{argmax} (r_0)$, set $\mathcal{N}_{1,<}$ and $\mathcal{N}_{1,>}$.}
    \STATE{ $V_1 =   r_0/\|r_0\|_2$.}
    \FOR{$k=1, \ldots,$ maxit}
        \STATE{$AV_k = [AV_{k-1}, Av_k]$ and store and expand $V_{k+1} = [V_k, v_{k+1}]$.} \;
        \IF{$k=1$}
        \STATE{[$Q,R$] = qr($AV|_{\mathcal{B}_1}$).}\;
        \STATE{ solve $AV|_{\mathcal{B}_1} y_k = (r_0)|_{\mathcal{B}_1}$ using $Q$ and $R$.}
        \STATE{$r_k= r_0-AV_k y_k$.}\;
        \ELSE
        \STATE{update $Q$ and $R$ using \eqref{qr_update_krylov_l1}.}\;
        \ENDIF
        \FOR{$i=1, \ldots,$ maxsimplexiters}
        \IF{$i=1$}
        \STATE{ $(\lambda_\mathcal{N},\mu_\mathcal{N} )_i = \begin{cases}
            (1,0) \quad  i \in \mathcal{N}_<,\\
            (0,1) \quad i \in \mathcal{N}_>,
          \end{cases}$ .}\; 
        \STATE{rhs = $- V_k^T A^T (\lambda_\mathcal{N}-\mu_\mathcal{N})$. } \;
        \ELSE{}
        \STATE{ update rhs using \eqref{update_rhs}.}\;
        \STATE{Update the $Q$ and $R$ factors  using \eqref{qr_update_basic_l1}. }\;
        \ENDIF
        \STATE{ solve    $ (AV_k |_{\mathcal{B}_k})^T z =$ rhs using $Q$ and $R$.}\;
        \STATE{$\lambda_{\mathcal{B}_k} = (1 + z)/2$  \quad        $\mu_{\mathcal{B}_k} = (1 - z)/2$}
        \IF{$ \min \lambda \ge 0  \quad \text{and} \quad\min \mu \ge 0$.}
        \STATE{     break;   optimal solution for $ \mathcal{K}_k$ .   }
        \STATE{  $x_k = x + Vy_k$.}
        \ELSE
        \STATE{ $  q =  \text{argmin}(\min \lambda_{ \mathcal{B}_k }, \min \mu_{ \mathcal{B}_k })$. }
        \ENDIF
        \STATE{ solve $AV_k |_{\mathcal{B}_k} d = e_j $  using $Q$ and $R$  where $j$ is index of $q$ in $\mathcal{B}_k$.}
        \STATE{$\alpha,r =  \min_{\mathcal{N}, \cdot  > 0} (r_k) / (AV_k d)$. }
        \STATE{ $y_k = y_k+ d\alpha \quad \text{and} \quad  r_k = r_k - AVd \alpha$.}\;
        \STATE{$\mathcal{B}_k =  \mathcal{B}_k \setminus \{q \} \cup \{r\}$ and update $\mathcal{N}_{k,<}$ and $\mathcal{N}_{k,>}$.}\;
        \ENDFOR
    \STATE{solve $AV_{k}  |_{\mathcal{B}_k} d  =  -Av_{k+1} |_{\mathcal{B}_k}$ using $Q$ and $R$. }
    \STATE{$\alpha_i =  (r_k)_i/ (AV_k d + A v_{k+1})_i$.}
    \IF{$    \text{sign}(r_k)|_\mathcal{N}^T  (AV_k d + A v_{k+1})|_\mathcal{N} > 0 $.}
        \STATE{ $ r = \text{argmin}_{i \in \mathcal{N}, \alpha_i  >0 } \alpha_i  $.}
        \ELSE
        \STATE{$r = \text{argmax}_{i \in \mathcal{N}, \alpha_i  < 0 }  \alpha_i$.}
        \ENDIF
        \STATE{$y_{k+1} = \begin{pmatrix} y_{k} + d\alpha \\ \alpha
          \end{pmatrix}       \quad \text{and}  \quad r_{k+1} = r_k -(AV_kd + Av_{k+1})\alpha$.}\;
      \STATE{$\mathcal{B}_{k+1} = \mathcal{B}_k \cup \{r\}$, update $\mathcal{N}_{k+1,<}$ and $\mathcal{N}_{k+1,<}$. }\;
      \ENDFOR
  \end{algorithmic}
  \caption{ Krylov-simplex for $\min_{x \in x_0 + \mathcal{K}_k} \|b-Ax\|_1$ \label{algo:l_1}}
\end{algorithm}

\section{Numerical Results} \label{sec:numerical results}
In this section we apply the two methods to some examples and discuss
the convergence behavior.

\subsection{Convergence bounds of GMRES}
The GMRES method minimizes $\|r_k\|_2$ over the Krylov subspace
method.  Recall that
\begin{lemma}
  Let $A \in \mathbb{R}^{n\times n}$, $b \in \mathbb{R}^n$ and $x^{\text{KS}_1}_k$ be the
  Krylov-Simplex iterates for $\ell_1$-norm, $x^{\text{KS}_\infty}_k$
  the iterates for Krylov-simplex for $\ell_\infty$-norm and
  $x^\text{GMRES}_k$ are the GMRES iterates for $\ell_2$ in the Krylov
  subspace $x_0 + \mathcal{K}_k(A,b)$.  We then have that
  \begin{equation}
    \| b- Ax^{\text{KS}_\infty}_k\|_\infty  \le    \|b-Ax^\text{GMRES}_k\|_2
  \end{equation}
  and
  \begin{equation}
    \| b- Ax^{\text{KS}_1}_k\|_1  \le    \sqrt{n}\|b-Ax^\text{GMRES}_k\|_2.
  \end{equation}
\end{lemma}

\begin{proof}
  Recall that $\| x\|_\infty \le \|x\|_2$.  We have then have 
  \begin{equation}
    \|b - Ax^{\text{GMRES}}_k\|_\infty \le \|b - Ax^\text{GMRES}_k\|_2
  \end{equation}
  The iterate $x^{\text{GMRES}}_k$ and $x^{\text{KS}_1}_k$ sit both in
  $x_0 + \mathcal{K}_k(A,b)$. But the latter is the result of a
  optimization over $x_0 + \mathcal{K}_k(A,b)$ in the
  $\ell_\infty$-norm. So
  \begin{equation}
    \| b- Ax^{\text{KS}_\infty}_k\|_\infty  \le    \|b-Ax^\text{GMRES}_k\|_\infty.
  \end{equation}
  Combining both inequalities leads to the result.  Similar argument
  holds for $\ell_1$ using $\|x\|_1 \le \sqrt{n} \| x\|_2$
\end{proof}

For a non-square matrix $A \in \mathbb{R}^{m \times n}$ the iterates
$x^{\text{KS}_1}_k$ and $x^{\text{KS}_\infty}_k$ are described using a
Golub-Kahan basis. Also the GMRES iterates on $x^\text{GMRES}_k$, now
using the normal equations $A^TAx =A^Tb$, are in the Krylov subspace
$x_0 + \mathcal{K}_k(A^T A,A^Tb)$.  We can then proof the same
inequalities as in the lemma above.

The lemma is useful for accelerating the convergence of a
Krylov-Simplex method.  Indeed, a good preconditioner for the GMRES
will also be a good preconditioner for a Krylov-simplex method.
\begin{figure}
  \begin{center}
    \includegraphics[width=0.9\textwidth]{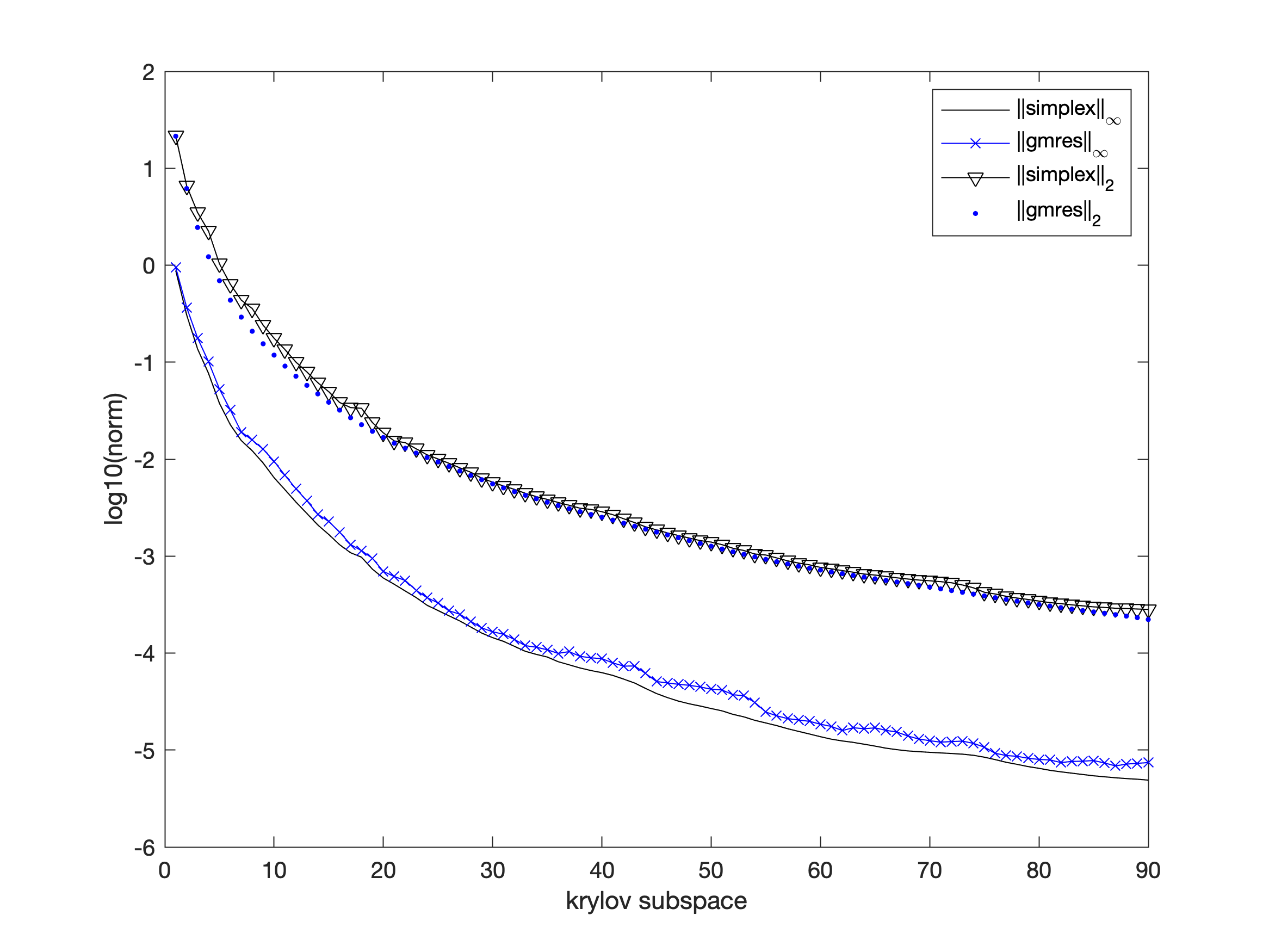}
  \end{center}
  \caption{We compare $\|r-Ax^{\text{KS}}_k\|_\infty$ and $\|r-Ax^{\text{GMRES}}_k\|_2$ where $x_k$ is
    optimized over the same Krylov subspace.  The first is optimized with Krylov-Simplex and the second with GMRES.
    The figure also shows the $\|r-Ax^{\text{KS}} _k\|_2$ and $\|r-Ax^{\text{GMRES}}_k\|_\infty$ and confirms the inequalities form the lemma  \label{fig:max2}}
\end{figure}
\begin{figure}
  \begin{center}
    \includegraphics[width=0.6\textwidth]{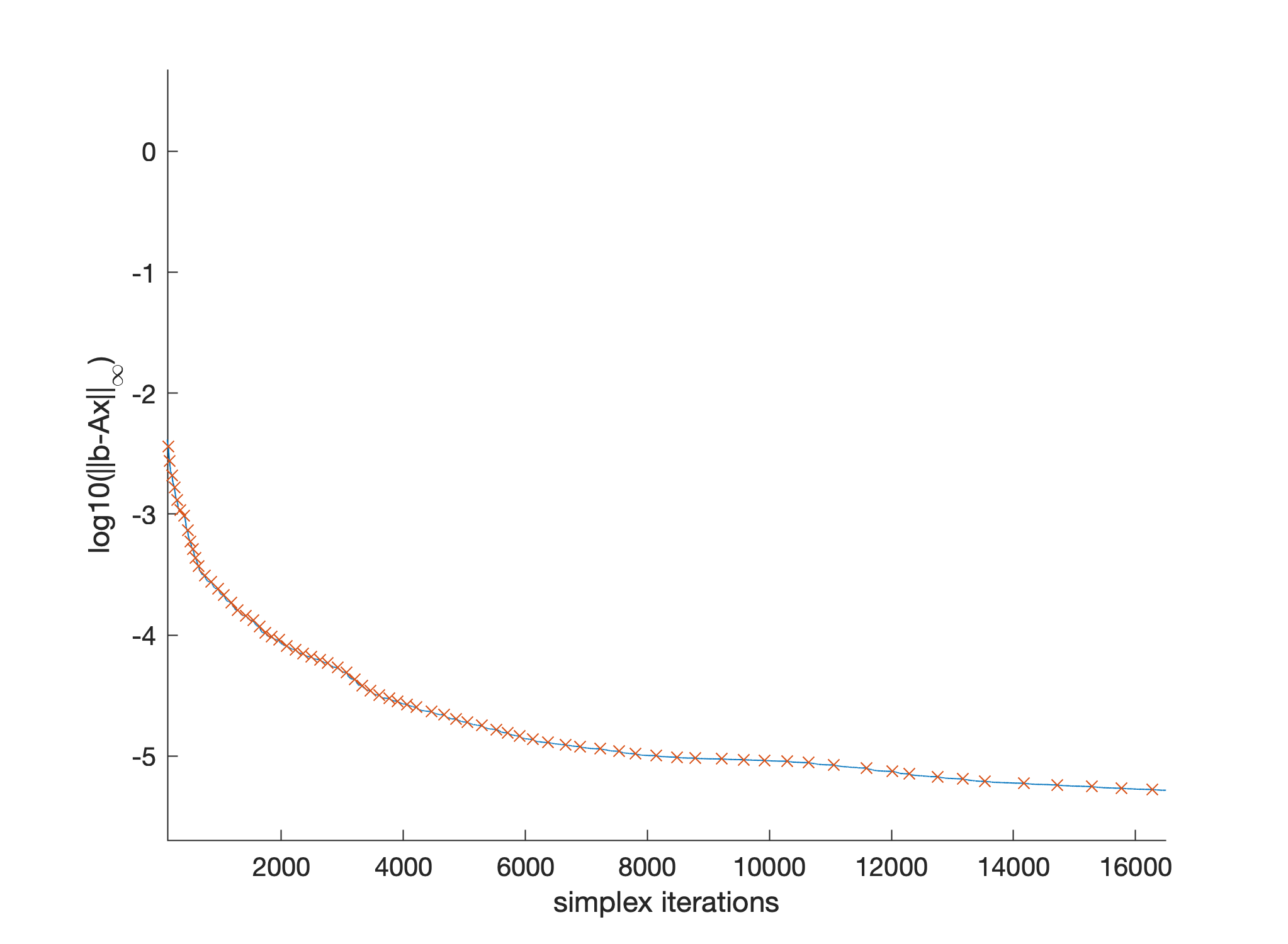}
  \end{center}
    \caption{ The convergence of $\|r-Ax\|_\infty$ as a function of the inner simplex iteration. The crosses indicate the points where the Krylov subspace is expanded \label{fig:max}}
\end{figure}

\subsection{$\ell_\infty$ example}
For a square matrix $A \in \mathbb{R}^{n \times n}$ we compare the
convergence of Krylov-Simplex in the $\ell_\infty$ norm and GMRES in
the $\ell_2$-norm. The matrix is symmetric represents a blurring operator
on a 125 by 125 grid.

For each solution in both iterative methods we can
calculate the the residuals in the $\ell_\infty$ and $\ell_2$ norm.
The figure \ref{fig:max} and \ref{fig:max2} shows that
\begin{equation}
  \|b-Ax_k^{\text{KS}_\infty}\|_\infty \le   \|b-Ax_k^{\text{GMRES}}\|_\infty
  \le \|b-Ax_k^{\text{GMRES}}\|_2 \le   \|b-Ax_k^{\text{KS}_\infty}\|_2
\end{equation}
Interesting is that the number of inner simplex iterations stay
limited to a few hundred each Krylov iteration.

\subsection{$\ell_1$ Example}
We illustrate the $\ell_1$ minimization of the residual with an image
deblurring example where some pixels are corrupted.

We take $N\times N$ image and generate a blurred image by weighting
each pixel with the stencils
\begin{equation}
  \begin{matrix}
     &1& \\
    1 &4& 1\\
    &1&
  \end{matrix}
  \quad \text{and} \quad
  \begin{matrix}
    1 & 1\\
    1 & 1
  \end{matrix}
\end{equation}
This generates a over determined data set $b$ of size $2N^2$, $N^2$ for
the first stencil and $N^2$ for the second stencil. We then corrupt
some data elements by replacing the data $b_i = b_i 0.999$. We create
a preconditioner by an incomplete Cholesky factorization of the normal
equation $A^TA$.  The Golub-Kahan basis is then generated with
$A L^{-1}$ as matrix operator, using a matrix-vector product with a
back-substitution of the application of $L^{-1}$.

In figure \ref{fig:l1} we show the convergence as function of the
number of inner simplex iterations and the corresponding convergence
as a function of the size of the Krylov subspace. We also show results
of a calculation where the inner simplex iterations are limited.

We typically see the staircase effect where the initial simplex inner
iteration after the Krylov subspace is expanded rapidly lead to and
improved basic set and convergence.  But then the benefit from
updating the basic set diminishes. The negative lagrange multipliers
become smaller and smaller and the gain from updating the basic set
becomes smaller. 

Also note, that in contrast to the $\ell_\infty$ example, the number
of inner simplex iteration can be significant, more then 20.000,
before the optimal basic set is found.  We therefore have also limited
the number of inner iterations and rather expand the Krylov subspace
than update the basic set, see figure \ref{fig:l1_2}.

\begin{figure}
  \begin{center}
    \includegraphics[width=0.6\textwidth]{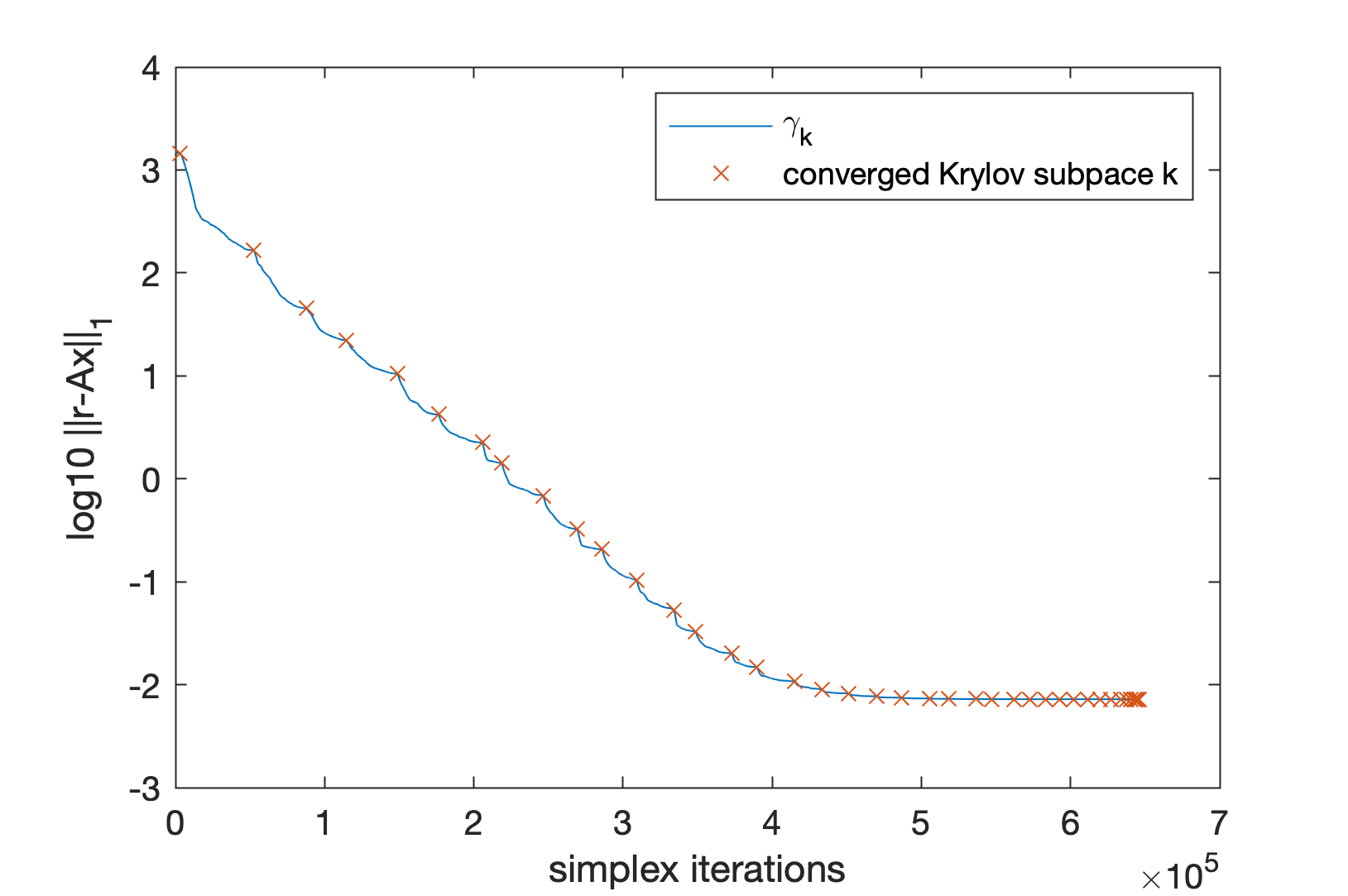}
  \end{center}
  \begin{center}
    \includegraphics[width=0.6\textwidth]{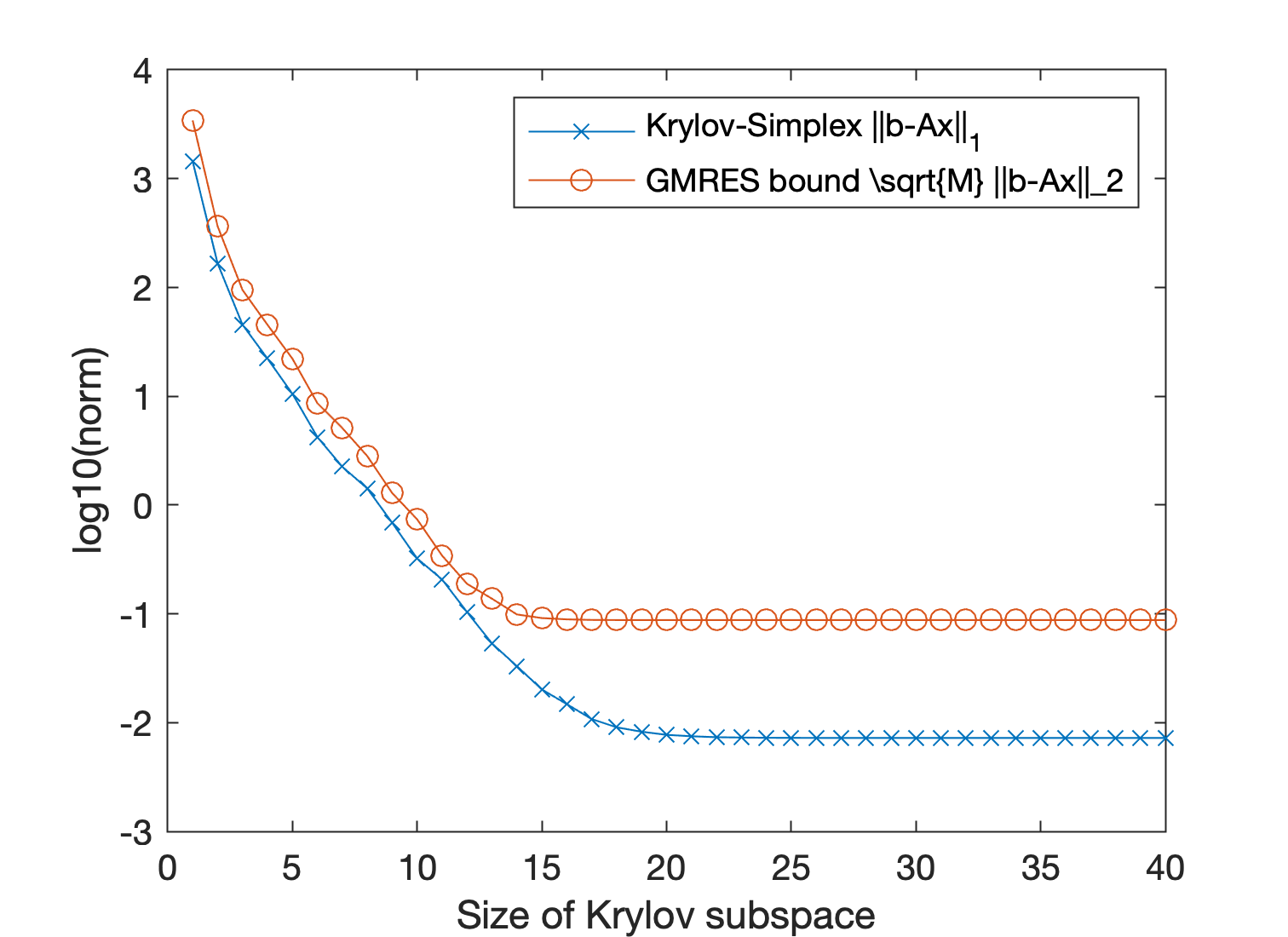}
  \end{center}
  \caption{ \textbf{Top.} The convergence of $\|b-Ax\|_1$ as a
    function of the inner simplex iterations. The convergence plateaus
    because the right hand side $b$, after corruption is not in the
    range of the $A$. The crosses indicate the optimal points for each
    Krylov subspace. \textbf{Bottom.} The convergence as a function of
    the size of the Krylov subspace. We also show the corresponding
    upper bound from GMRES applied to the Normal equations, based on
    the lemma. \label{fig:l1}}
\end{figure}
\begin{figure}
  \begin{center}
    \includegraphics[width=0.6\textwidth]{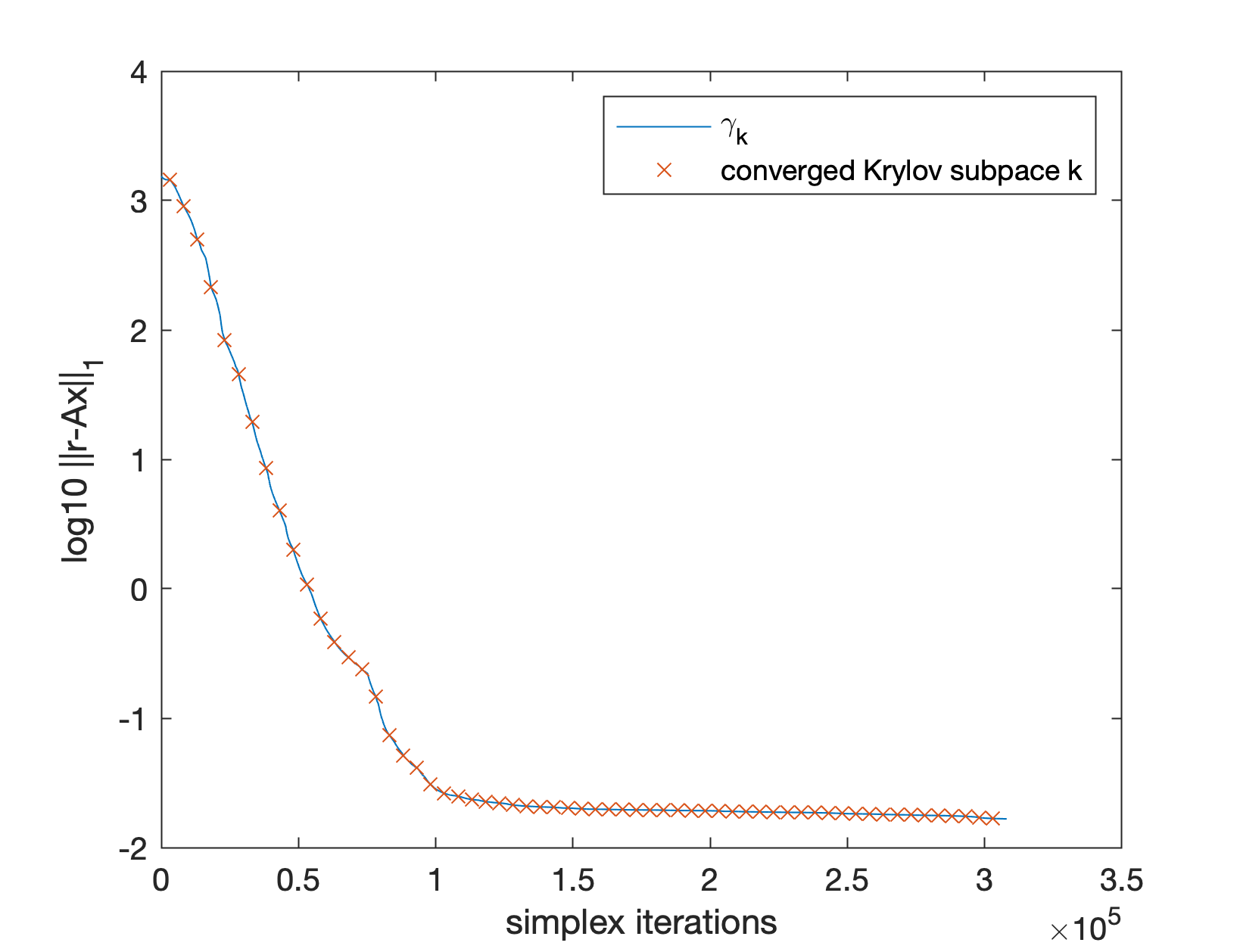}
  \end{center}
  \begin{center}
    \includegraphics[width=0.6\textwidth]{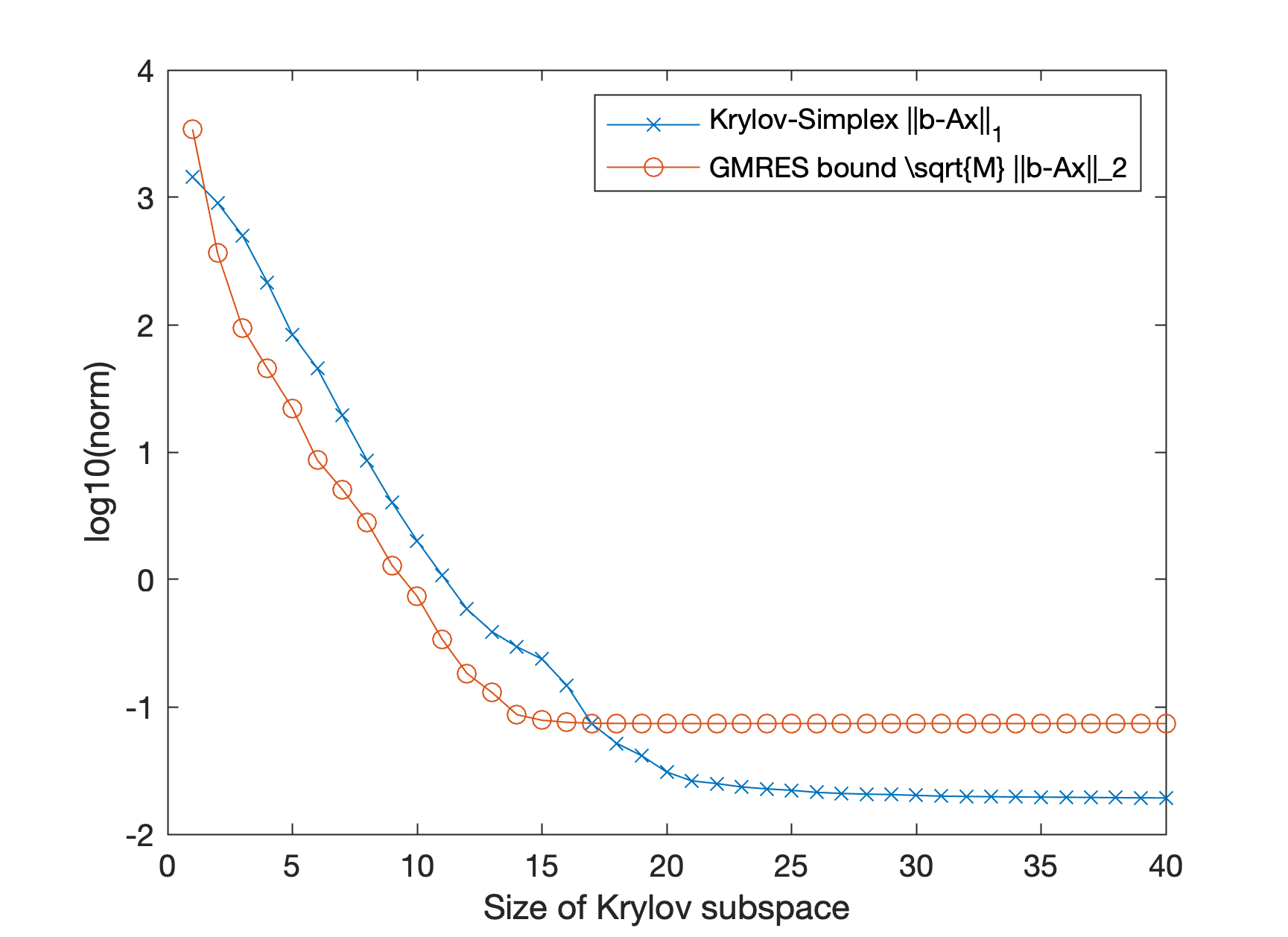}
  \end{center}
  \caption{ \textbf{Top.} We now limit the inner simplex iterations to
    8000.  We see that mainly benefit from the initial convergence
    within each Krylov subspace.  \textbf{Bottom.} Since we do not
    fully optimize within each Krylov subspace, we are not below the
    GMRES bound. However, eventually we drop below the
    bounds. \label{fig:l1_2}}
\end{figure}

\section{Discussion and Conclusions}
In this paper we explore the solution of the projected LP problems
that appear when we minimize the residual, $b-Ax$, over a Krylov
subspace in the $\ell_1$ or $\ell_\infty$-norm. This gives a series of
growing LP problems for the coefficients $y_k$ that describe the
solution in each Krylov subspace.  These are LP problems with dense
matrices that are solved with a primal simplex method where the
initial basic set can be easily generated from the previous Krylov
iteration.  An inner loop then updates this basic set until it is
optimal for the current Krylov subspace.  When the basic set changes,
the basic matrix is changed by a single row.  We keep its a QR
factorization up to date during the full algorithm with QR updates
only.

Numerical experiments show that the inner simplex does not need to be
fully converged before the Krylov subspace can be expanded.  This
limits the number of inner simplex iterations and only slightly delays
the Krylov convergence.  As the Krylov subspace expands we have to
find additional coefficients $y_k$ and update the existing once, hence
our choice for the primal simplex method.

At this point we have not explicitly exploited the Krylov structure.
We currently store both $AV_k$ and $V_k$.  The storage of $AV_k$ can
be avoided by exploiting the Arnoldi or bi-diagonalization relations
and using an additional multiplication with a small dense matrix.
This is a topic of future research.

Another topic of future research is the parallel implementation. Due
to efficient usage of QR updates, the most time consuming part in the
current research implementation is the sparse-matrix vector product,
the update of search direction with $AV_k d$ and finding the optimal
step size $\alpha$ using a \texttt{max}/\texttt{min} over the long
vector.  All these parts can be easily parallelized over a large
number of cores and computer.

A significant drawback of the method, compared to for example GMRES, is
that we need to calculate  $AV_k d$, see algorithm \ref{stepsizesimplex},  each inner simplex iteration to be able to check the bounds
and determine the step size. This is a multiplication of $AV_k$ or
$V_{k+1} H_{k+1}$ with $d \in \mathbb{R}^k$.  In contrast, GMRES only
requires $H_{k+1,1}$ to determine the coefficients.  Although this can
be executed in parallel it is a significant limitation of the proposed
method.

As a result the current prototype is not competitive with direct
solution of the LP problem with interior point methods. The second
example takes about 10sec to solve with an off-the shelve
interior-point LP solver, while our method takes about 60 sec.

For the $\ell_\infty$ example the number of inner simplex iteration is
limited. For $\ell_1$ the total number of inner simplex iteration can
be significant, up to 100k iterations.

We have named this a Krylov-Simplex method because there is an outer
loop where the Krylov method is expanded and inner loop where the
simplex method solves the projected LP problem.  In analogy to
Newton-Krylov where an outer Newton iteration is combined with a
Krylov inner iteration.

In the future we will also explore if column generation to find the
optimal vector to expand the subspace that represents the solution.
Also replacing the simplex method with an interior-point method is
wort exploring.  It is then not possible to warm-start from the
previous Krylov subspace, but the search for the optimal $y_k$ might
be faster because we need as many inner iterations that each require
the calculation of $AV_kd$.

\section*{Acknowledgments}
Wim Vanroose is thankful to Michael Saunders, Matthias Bollh\"ofer,
Julian Hall and Jacek Gondzio for feedback and fruitful discussions.

\bibliographystyle{siam}
\bibliography{references}
\end{document}